\documentclass[a4paper,12pt]{amsart}
\usepackage{amsmath, amssymb, amsthm, xcolor}
\usepackage{verbatim}
\usepackage{xypic}
\usepackage{graphicx}
\usepackage{enumerate}
\usepackage[colorlinks=true, citecolor=blue, linkcolor=blue, bookmarks=true]{hyperref}

\input{xy} 
\xyoption{all} 
\CompileMatrices

\newcounter{theorem}
\newtheorem{theorem}[theorem]{Theorem}
\newtheorem{lemma}[theorem]{Lemma}
\newtheorem{proposition}[theorem]{Proposition}
\newtheorem{corollary}[theorem]{Corollary}
\newtheorem*{conjecture*}{Conjecture}

\newtheorem{introtheorem}[theorem]{Theorem}

\newtheorem{introcorollary}[theorem]{Corollary}

\theoremstyle{definition}
\newtheorem{definition}[theorem]{Definition}

\theoremstyle{remark}
\newtheorem*{remark*}{Remark}
\newtheorem{remark}[theorem]{Remark}

\numberwithin{equation}{section}

\newcommand{\Z}{\mathcal Z}
\newcommand{\K}{\mathbb{K}}
\newcommand{\N}{\mathbb{N}}
\newcommand{\dimnuc}{\dim_{\mathrm{nuc}}}
\newcommand{\dr}{\mathrm{dr}}

\title[Nuclear dimension of stably projectionless $\mathrm{C}^*$-algebras]{Nuclear dimension of simple stably projectionless $\mathrm{C}^*$-algebras}
\author[J.\ Castillejos]{Jorge Castillejos}
\address{\hskip-\parindent Jorge Castillejos, Department of Mathematics, KU Leuven, Celestijnenlaan 200b, 3001 Leuven, Belgium}
\email{jorge.castillejoslopez@kuleuven.be}
\author[S.\ Evington]{Samuel Evington}
\address{\hskip-\parindent Samuel Evington, Institute of Mathematics, Polish Academy of Sciences, ul. {\'S}niadeckich 8, 00-656 Warszawa, Poland}
\curraddr{School of Mathematics and Statistics, University of Glasgow, Glasgow, G12 8QW, Scotland}
\email{Samuel.Evington@glasgow.ac.uk}

\thanks{SE is supported by NCN (2014/14/E/ST1/00525) and EPSRC (EP/R025061/1); JC is partially supported by European Research Council Consolidator Grant 614195 RIGIDITY, and by a long term structural funding -- 
a Methusalem grant of the Flemish Government.}

\begin{document}
\maketitle

\begin{abstract}
We prove that $\Z$-stable, simple, separable, nuclear, non-unital $\mathrm{C}^*$-algebras have nuclear dimension at most $1$. This completes the equivalence between finite nuclear dimension and $\mathcal{Z}$-stability for simple, separable, nuclear, non-elementary $\mathrm{C}^*$-algebras.
\end{abstract}

\renewcommand*{\thetheorem}{\Alph{theorem}}

\section*{Introduction}

The Elliott Classification Programme, a 40-year endeavour involving generations of researchers, asks the following question: when are K-theory and traces a complete invariant for simple, separable, nuclear C$^*$-algebras? 

Fundamentally, there are two cases to consider: the \emph{unital} case and the \emph{stably projectionless} case. (This dichotomy follows from Brown's Theorem (\cite{Br77}) and is discussed further below.) Recall that a C$^*$-algebra $A$ is said to be stably projectionless if there are no non-zero projections in the matrix amplification $M_n(A)$ for any $n \in \mathbb{N}$. Stably projectionless, simple, separable, nuclear C$^*$-algebras arise naturally as crossed products (\cite{KK96}), and can also be constructed using inductive limits with a wide variety of K-theoretic and tracial invariants occurring (\cite{Bl80, Ra02,Ts05,Ja13,Go16,Go17,EGLN17,EGLN:gtr1}). 

In the unital case, a definitive answer for when K-theory and traces form a complete invariant is now known (\cite{Ki95,Phi00,GLN15,EGLN15,TWW17,Wi14}). Firstly, Rosenberg and Schochet's universal coefficient theorem (\cite{RS:UCT}) must hold for the C$^*$-algebras concerned. Secondly, the C$^*$-algebras must have \emph{finite nuclear dimension} (\cite{WZ10}). This second condition has a geometric flavour and generalises the notation of finite covering dimension for topological spaces. Recent results (\cite{Go16,E17,EGLN17,Go17}) are now converging on a similar classification result in the stably projectionless case; the state of the art will be discussed below.


A major programme of research now focuses on providing methods for verifying finite nuclear dimension in practice. In the unital setting, a recent result of the authors together with Tikuisis, White and Winter (\cite{CETWW}) shows that finite nuclear dimension can be accessed through the tensorial absorption condition known as \emph{$\Z$-stability}, where $\Z$ is the \emph{Jiang--Su algebra} (discussed in more detail below). 

In concrete examples, it can be very hard to prove directly that a C$^*$-algebra has finite nuclear dimension. The strategy of verifying $\Z$-stability instead has recently been used to prove that certain unital, simple, separable, nuclear C$^*$-algebras coming from dynamical systems are classifiable (\cite{CJKMST17,KS18}). However, since this strategy relies on \cite{CETWW}, it has until now only been available in the unital setting. 

In this paper, we consider and overcome the conceptual and technical challenges unique to the non-unital setting, allowing us prove the following: 

\begin{introtheorem}\label{thm:NewMain2}
	Let $A$ be a simple, separable, nuclear, $\Z$-stable $\mathrm{C}^*$-algebra. Then $A$ has nuclear dimension at most 1. 
\end{introtheorem}

For the following reasons, the non-unital case is harder than the unital case and needs new methods. Obviously, we cannot just unitise our C$^*$-algebras because this breaks both simplicity and $\Z$-stability. A more fundamental issue is that non-unital, simple C$^*$-algebras need not actually be \emph{algebraically simple}. There can be non-trivial (non-closed) ideals. Examples of such ideals are the domains of \emph{unbounded traces}, which may now exist and must therefore be taken into account. Furthermore, \cite{CETWW} builds on the foundations of \cite{BBSTWW}, which has a global assumption of unitality and makes explicit use of the unit at a number of critical points in the argument (an example is the $2 \times 2$ matrix trick of \cite[Section 2]{BBSTWW}, which is inspired by ideas of Connes). 

To understand how we circumvent the issues associated to unbounded traces, it will be helpful to first discuss the folklore result, alluded to above, that Brown's Theorem (\cite{Br77}) implies a dichotomy for simple C$^*$-algebras between the unital and the stably projectionless case.

Writing $\K$ for the C$^*$-algebra of compact operators (on a separable, infinite-dimensional Hilbert space), recall that C$^*$-algebras $A,B$ are \emph{stably isomorphic} if $A \otimes \K \cong B \otimes \K$. Suppose now that $A$ is a simple, separable C$^*$-algebra that is not stably projectionless. Then there exists a non-zero projection $p \in A \otimes \K$, and so the hereditary subalgebra $p(A \otimes \K)p$ is unital. By \cite[Theorem 2.8]{Br77}, $p(A \otimes \K)p$ is stably isomorphic to $A \otimes \K$, and hence stably isomorphic to $A$ (see Section \ref{sec:Reductions} for more details).   

Crucial to proving Theorem \ref{thm:NewMain2} in general is the observation that the hypotheses and the conclusion depend only on the stable isomorphism class of $A$.\footnote{i.e. (a) $A$ is a simple, separable, nuclear and $\Z$-stable $\mathrm{C}^*$-algebra if and only if $A \otimes \mathbb{K}$ is likewise, and (b) $\dim_{\mathrm{nuc}} A \leq 1$ if and only if $\dim_{\mathrm{nuc}} (A \otimes \mathbb{K}) \leq 1$; see Proposition \ref{thm:StableIsoInvariants} for details and references.} Hence, by \cite[Theorem B]{CETWW} and the folklore result above based on Brown's Theorem, it suffices to prove Theorem \ref{thm:NewMain2} in the stably projectionless case.  

However, this folklore reduction is not enough for us. We go a step further and pass to a hereditary subalgebra $A_0 \subseteq A \otimes \K$ on which all tracial functionals are bounded and the set of tracial states $T(A_0)$ is weak$^*$ compact. The existence of such hereditary subalgebra follows from the Cuntz semigroup computation of \cite{ERS11} for $\Z$-stable $\mathrm{C}^*$-algebras, and Brown's Theorem assures us that $A_0$ is stably isomorphic to $A$. This second reduction puts us in a position where a similar proof strategy to that of \cite{BBSTWW} can be implemented, and where the key new ingredient from \cite{CETWW}, \emph{complemented partitions of unity} (CPoU), is also available.

Of course, we still have to deal with the global assumption of unitality in \cite{BBSTWW}. A key tool in this endeavour is our unitisation lemma for order zero maps into ultrapowers (Lemma \ref{lem:Unitisation}), which allows us to assume the domains of certain maps are unital in places where simplicity and $\Z$-stability are only really needed on the codomain side in \cite{BBSTWW}.  
\\
\\ 
We now turn to the broader context of Theorem \ref{thm:NewMain2} and its applications. As alluded to above, \emph{nuclear dimension} for C$^*$-algebras is a non-commutative dimension theory that reduces to the covering dimension of the spectrum in the commutative case. \emph{Finite nuclear dimension} has proven to be a technically useful strengthening of nuclearity, that is both necessary for classification (\cite{Vi99,Ro03,To08,GK14}) and a vital ingredient of the most recent classification theorems (\cite{Ki95,Phi00,GLN15,EGLN15,TWW17,Wi14}).  

The \emph{Jiang--Su algebra} $\Z$ (\cite{JS99}) is a simple $\mathrm{C}^*$-algebra, which plays a fundamental role in the classification of simple C$^*$-algebras since $A$ and $A \otimes \Z$ have the same K-theory and traces under mild hypotheses. A C$^*$-algebra is said to be \emph{$\Z$-stable} if $A \cong A \otimes \Z$. Moreover, any C$^*$-algebra can be $\Z$-stabilised by tensoring with the Jiang--Su algebra because $\Z \cong \Z \otimes \Z$. In many ways, the Jiang--Su algebra is the $\mathrm{C}^*$-algebraic analogue of the hyperfinite $\mathrm{II}_1$ factor $\mathcal{R}$ (\cite{MvN43}), with $\Z$-stability analogous to the McDuff property (\cite{McD70}).  

Combining Theorem \ref{thm:NewMain2} with the main results of \cite{Wi12} and \cite{Ti14}, we arrive at the following relationship between finite nuclear dimension and $\Z$-stability, which was conjectured by Toms--Winter (\cite{To09}).

\begin{theorem}\label{thm:NewMain}
	Let $A$ be a non-elementary, simple, separable, nuclear $\mathrm{C}^*$-algebra. The following are equivalent:
\begin{enumerate}[(i)]
	\item $A$ has finite nuclear dimension;
	\item $A$ is $\Z$-stable.
\end{enumerate}
\end{theorem}

One striking consequence of Theorems \ref{thm:NewMain2} and \ref{thm:NewMain} is that nuclear dimension can only attain three different values in the simple setting.

\begin{introcorollary}\label{cor:NewTrichotomy}
	The nuclear dimension 
	of a simple $\mathrm{C}^*$-algebra is $0, \, 1$ or $\infty$.
\end{introcorollary}

This is in stark contrast to the commutative case, where all non-negative integers can occur. Moreover, we remark that the $\mathrm{C}^*$-algebras of nuclear dimension zero are known to be precisely the approximately finite dimensional $\mathrm{C}^*$-algebras (\cite[Remark 2.2.(iii)]{WZ10}).\footnote{In the non-separable case, there are different and non-equivalent definitions of approximately finite dimensional (\cite{FK10}). The one required here is that any finite set is approximately contained in a finite dimensional subalgebra; see for example
\cite[Definition 2.2]{Cas17}.}   

Whilst Corollary \ref{cor:NewTrichotomy} is interesting in its own right, we believe the main applications of our results will be in classification of simple, stably projectionless C$^*$-algebras. Theorem \ref{thm:NewMain2} opens up a new pathway to proving that concrete examples of stably projectionless, simple, separable, nuclear C$^*$-algebras, such as C$^*$-algebras coming from flows on C$^*$-algebras or from actions of more general locally compact groups, have finite nuclear dimension: it now suffices to verify $\Z$-stability. 

We end this introduction with a discussion of the state of the art for the classification of simple, stably projectionless C$^*$-algebras. As mentioned above, there has been impressive progress in recent years (\cite{Go16,Go17, E17}). As in the unital case, the classification is via a functor constructed from the K-theory and the tracial data of the $\mathrm{C}^*$-algebra; this functor is called the \emph{Elliott invariant} and is typically denoted $\mathrm{Ell}(\cdot)$ (see \cite[Definition 2.9]{Go17} for a precise definition).

By combining Theorem \ref{thm:NewMain2} with \cite[Theorem 1.2]{Go17}, one obtains a classification of simple, separable, nuclear $\mathrm{C}^*$-algebras in the UCT class that tensorially absorb the C$^*$-algebra $\mathcal{Z}_0$ -- a stably projectionless analogue of the Jiang--Su algebra introduced in \cite[Definition 8.1]{Go17}.

\begin{corollary}\label{cor:Newclassification}
	Let $A$ and $B$ be simple, separable, nuclear $\mathrm{C}^*$-algebras which satisfy the UCT. Then 
	\begin{equation}
	A \otimes \mathcal{Z}_0 \cong B \otimes \mathcal{Z}_0 \text{ if and only if } \mathrm{Ell}(A \otimes \mathcal{Z}_0) \cong \mathrm{Ell}(B \otimes \mathcal{Z}_0). \notag
	\end{equation}
\end{corollary}

Corollary \ref{cor:Newclassification} reduces to the celebrated Kirchberg-Phillips classification (\cite{Ki95,Phi00}) in the traceless case and is otherwise a result about stably projectionless C$^*$-algebras. For these C$^*$-algebras, the difference between $\mathcal{Z}_0$-stability and $\mathcal{Z}$-stability, roughly speaking, comes down to how complex the interaction between the K-theory and traces is allowed to be; see \cite{Go17} for more details.   

\subsection*{Structure of Paper}

Section 1 reviews the necessary preliminary material as appropriate to the non-unital setting. Section 2 is concerned with the invariance of $\mathrm{C}^*$-algebraic properties under stable isomorphism and the reduction argument outlined above. The next three sections generalise the necessary technical machinery from \cite{BBSTWW} and \cite{CETWW}. Section 3 concerns the existence of an order zero embedding $\Phi:A \rightarrow A_\omega$ with appropriate finite dimensional approximations. Section 4 contains the aforementioned unitisation lemma for order zero maps into ultrapowers. Section 5 is devoted to a uniqueness theorem for maps into ultrapowers, which we shall use to compare (unitisations of)  $\Phi$ and the canonical embedding $A \rightarrow A_\omega$. Theorem \ref{thm:NewMain2} and its corollaries are proved in Section 6, with analogous results for decomposition rank (a forerunner of nuclear dimension) proved in Section 7. Since some preliminary lemmas from \cite{BBSTWW} are stated only in the unital case, we include an appendix with their non-unital versions. 

\subsection*{Acknowledgements} Part of this work was undertaken during a visit of JC to IMPAN. 
JC thanks SE and IMPAN for their hospitality. SE would like to thank George Elliott for his helpful comments on this research during SE's secondment at the Fields Institute, which was supported by the EU RISE Network \emph{Quantum Dynamics}.   
The authors would also like to thank Jamie Gabe, G\'abor Szab\'o, Stefaan Vaes and Stuart White for useful comments on an earlier version of this paper.

\numberwithin{theorem}{section}

\section{Preliminaries}

In this section, we recall the most important definitions and results that will be used in the sequel, and we introduce the notation used in this paper.  

We write $\mathbb{K}$ to denote the $\mathrm{C}^*$-algebra of compact operators (on a separable, infinite-dimensional Hilbert space). Given a $\mathrm{C}^*$-algebra $A$, we write $A_+$ for the positive elements of $A$ and $A_{+,1}$ for the positive contractions; we write $\mathrm{Ped}(A)$ for the Pedersen ideal of $A$, which is the minimal dense ideal of $A$ (see \cite[Section 5.6]{Ped79}); and we write $A^\sim$ for the unitisation of $A$. Our convention is that, if $A$ is already unital, then we adjoin a new unit, so $A^\sim \cong A \oplus \mathbb{C}$ as $\mathrm{C}^*$-algebras. For $S \subseteq A$ self-adjoint, we set $S^\perp := \lbrace a \in A: ab = ba = 0, \forall b \in S\rbrace$.  For $\epsilon>0$ and $a,b \in A$, the notation $a \approx_{\epsilon} b$ means $\|a - b\|< \epsilon$. For  $a,b \in A$ with $b$ self-adjoint, we write $a \vartriangleleft b$ to mean that $ab= ba = a$. 
 
We use the common abbreviation c.p.c.\ for completely positive and contractive maps between $\mathrm{C}^*$-algebras. 
A c.p.c.\ map $\phi:A \rightarrow B$ is \emph{order zero} if it preserves orthogonality in the sense that, for $a,b \in A_+$, $\phi(a)\phi(b) = 0$ whenever $ab = 0$.

Following \cite[Definition 2.1]{WZ10}, a $\mathrm{C}^*$-algebra $A$ has \emph{nuclear dimension at most $n$}, if there is a net $(F_i, \psi_i,\phi_i)_{i \in I}$, where $F_i$ is a finite dimensional $\mathrm{C}^*$-algebra, $\psi_i: A \to F_i$ is a c.p.c.\ map, and $\phi_i: F_i \to A$ is a c.p.\ map, such that $\phi_i \circ \psi_i (a) \to a$ for all $a \in A$ and, moreover, each $F_i$ decomposes into $n+1$ ideals $F_i = F_i^{(0)} \oplus \cdots \oplus F_i^{(n)}$ for which the restrictions $\phi_i|_{F_i^{(k)}}$ are c.p.c.\ order zero. The \emph{nuclear dimension} of $A$, denoted by $\dim_{\mathrm{nuc}} A$, is defined to be the smallest such $n$ (and to be $\infty$, if no such $n$ exists).
The \emph{decomposition rank}, a forerunner of nuclear dimension, is obtained if one additionally requires $\phi_i$ to be a c.p.c.\ map \cite[Definition 3.1]{KW04}. We shall denote the decomposition rank of a $\mathrm{C}^*$-algebra $A$ by $\mathrm{dr}(A)$.    
 
By a \emph{trace} on a $\mathrm{C}^*$-algebra we will typically mean a tracial state, i.e.\ a positive linear functional $\tau:A \rightarrow \mathbb{C}$ of operator norm $1$ such that $\tau(ab)=\tau(ba)$ for all $a, b \in A$. We write $T(A)$ for the set of tracial states on $A$ endowed with the weak$^*$-topology. More general notions of traces are discussed in Section \ref{sec:GenTraces} below.

By a \emph{cone} we will mean a convex subset $C$ of a locally convex space that satisfies $C + C \subseteq C$, $\lambda  C \subset C$ for $\lambda > 0$, and $C \cap (-C) = \lbrace 0 \rbrace$. A \emph{base} for a cone $C$ is a closed, convex, and bounded subset $X$ such that for any non-zero $c \in C$ there exist unique $\lambda > 0$ and $x \in X$ such that $c = \lambda x$. By \cite[Theorem II.2.6]{Alf71}, a cone is locally compact if and only if it has a compact base. 
A map $f: C \rightarrow D$ between cones is \emph{linear} if $f(\lambda x + \mu y) = \lambda f(x) + \mu f(y)$ for $\lambda, \mu \geq 0$ and $x,y \in C$. If $X$ is a compact base for the cone $C$, then any continuous affine map $X \rightarrow D$ extends uniquely to a continuous linear map $C \rightarrow D$. 

\subsection{Generalised Traces}\label{sec:GenTraces}

In this preliminary section, we briefly discuss the generalisations of traces that arise in the general theory of $\mathrm{C}^*$-algebras. 

\begin{definition}[{cf.\ \cite[Definition 2.22]{BK04}}]
	A \emph{quasitrace}\footnote{Strictly speaking, a \emph{2-quasitrace}; however, we shall not need this 
	terminology.} on a $\mathrm{C}^*$-algebra $A$ is a function $\tau:A_+ \rightarrow [0,\infty]$ with $\tau(0) = 0$ such that 
	\begin{itemize}
		\item[(i)] $\tau(a^*a) = \tau(aa^*)$ for all $a \in A$,
		\item[(ii)] $\tau(a + b) = \tau(a) + \tau(b)$ for all commuting elements $a,b \in A_+$,
		\item[(iii)] $\tau$ extends to a function $\tau_2:M_2(A)_+ \rightarrow [0,\infty]$ for which (i) and (ii) hold.
	\end{itemize}
	The quasitrace $\tau$ is \emph{additive} if (ii) holds for all $a,b \in A_+$.\footnote{We use the terminology \emph{additive quasitrace} because we are reserving the word \emph{trace} for tracial states. For additive quasitraces, condition (iii) is automatic with $\tau_2$ given by the usual formula.} Setting $\mathrm{Dom}_{1/2}(\tau) := \lbrace a \in A: \tau(a^*a) < \infty \rbrace$, we say that $\tau$ is \emph{densely-defined} if $\mathrm{Dom}_{1/2}(\tau)$ is dense in $A$, and that $\tau$ is \emph{bounded} if $\mathrm{Dom}_{1/2}(\tau) = A$. 
\end{definition}

We write $Q\widetilde{T}(A)$ for the cone of densely-defined, lower-semicontinuous quasitraces; $\widetilde{T}(A)$ for the cone of densely-defined, lower-semicontinuous, additive quasitraces; and $\widetilde{T}_b(A)$ for the cone of bounded, additive quasitraces. The topology on these cones is given by pointwise convergence on $\mathrm{Ped}(A)$.   

Since the traces on a $\mathrm{C}^*$-algebra will play a crucial role in the arguments of this paper, the following existence theorem of Blackadar--Cuntz is fundamental. 

\begin{theorem}[{\cite[Theorem 1.2]{BC82}\label{thm:TracialStatesExist}}]
	Let $A$ be a simple $\mathrm{C}^*$-algebra such that $A \otimes \K$ contains no infinite projections. Then $Q\widetilde{T}(A) \neq 0$. 
\end{theorem}

It is an open question whether $Q\widetilde{T}(A) = \widetilde{T}(A)$ in general. However, when $A$ is exact, this is a famous result of Haagerup; see \cite{Ha14} for the unital case and \cite[Remark 2.29(i)]{BK04} for how to deduce the general case from \cite{Ha14}.

Every $\tau \in Q\widetilde{T}(A)$ has a unique extension to a densely-defined, lower-semicontinuous quasitrace on $A \otimes \K$, which is additive whenever $\tau$ is additive \cite[Remark 2.27(viii)]{BK04}. Therefore, we have canonical isomorphisms $Q\widetilde{T}(A) \cong Q\widetilde{T}(A \otimes \K)$ and $\widetilde{T}(A) \cong \widetilde{T}(A \otimes \K)$, which we treat as identifications. Furthermore, every $\tau \in \widetilde{T}_b(A)$ has a unique extension to a positive linear functional on $A$, which we also denote $\tau$, satisfying the trace condition $\tau(ab) = \tau(ba)$ for all $a,b \in A$.

Let $a,b \in A_+$. If there exists a sequence  $(x_n)_{n\in\N}$ in $A$ such that $b = \sum_{n=1}^\infty x_n^*x_n$ and $\sum_{n=1}^\infty x_nx_n^* \leq a$, then $b$ is said to be \emph{Cuntz--Pedersen subequivalent} to $a$; see \cite{Cu79}. Our notation for this subequivalence will be $b \preccurlyeq  a$. The following proposition is proven by the same method as \cite[Proposition 4.7]{Cu79}. For the benefit of the reader, we give full details.

\begin{proposition}\label{prop:ExtendingTraces}
Let $A$ be a simple, separable $\mathrm{C}^*$-algebra and $B \subseteq A$ a non-zero hereditary subalgebra. The restriction map $\rho:\widetilde{T}(A) \rightarrow \widetilde{T}(B)$ is a linear homeomorphism of cones.
\end{proposition}
\begin{proof}
	Since $\mathrm{Ped}(B) \subseteq \mathrm{Ped}(A)$, the restriction of a densely-defined quasitrace on $A$ is a densely-defined quasitrace on $B$. Restriction also preserves additivity and lower-semicontinuity. Hence, $\rho$ is well defined. Continuity of $\rho$ follows immediately from the fact that $\mathrm{Ped}(B) \subseteq \mathrm{Ped}(A)$, and it is clear that $\rho$ is linear.
	
We now turn to proving that $\rho$ is surjective. Let $\sigma \in  \widetilde{T}(B)$. Define $\tau:A_+ \rightarrow [0,\infty]$ by $\tau(a) := \sup \lbrace \sigma(b): b \in B_+, b \preccurlyeq a \rbrace$. The following properties of $\tau$ are easy to verify
\begin{align}
	\tau(0) &= 0,\\
	\tau(a^*a) &= \tau(aa^*), &a \in A,\\ 
	\tau(\lambda a) &= \lambda \tau(a), &\lambda \geq 0, \ a \in A_+,\\
	\tau(a_1 + a_2) &\geq \tau(a_1) + \tau(a_2), &a_1,a_2 \in A_+. 
\end{align}
Let $a_1,a_2 \in A_+$. Suppose $b \in B_+$ and $b \preccurlyeq a_1 + a_2$. By \cite[Corollary 1.2]{Pe69}, there exist $b_1,b_2 \in A_+$ with $b = b_1 + b_2$ such that $b_1 \preccurlyeq a_1$ and $b_2 \preccurlyeq a_2$. Since $B$ is a hereditary subalgebra, $b_1,b_2 \in B_+$. Hence, 
\begin{equation}
	\sigma(b) = \sigma(b_1) + \sigma(b_2) \leq \tau(a_1) + \tau(a_2).
\end{equation} 
Taking the supremum, we get $\tau(a_1 + a_2) \leq \tau(a_1) + \tau(a_2)$. Therefore, we have $\tau(a_1 + a_2) = \tau(a_1) + \tau(a_2)$. This completes the proof that $\tau$ is an additive quasitrace. 

Since $B$ is a hereditary subalgebra of $A$, the restriction of the Cuntz--Pedersen subequivalence relation on $A$ to $B$ is the same as the Cuntz--Pedersen subequivalence relation on $B$. It follows that $\tau|_{B_+}$ is $\sigma$. As $\sigma$ is densely defined on $B$ and $A$ is simple, $\tau$ is densely defined.

Let $\widetilde{\tau}(a) := \sup_{\epsilon > 0} \tau((a - \epsilon)_+)$ be the \emph{lower-semicontinuous regularisation} of $\tau$ (see \cite[Remark 2.27.(iv)]{BK04} and \cite[Lemma 3.1]{ERS11}). 
Then $\widetilde{\tau}$ is a densely-defined, lower-semicontinuous, additive quasitrace on $A$, and we still have $\widetilde{\tau}|_{B_+} = \sigma$ because $\sigma$ is lower-semicontinuous. Therefore, $\rho$ is surjective. 

We now prove that $\rho$ is injective. Let $\sigma$, $\tau$, and $\widetilde{\tau}$ be as above. Suppose $\psi \in \widetilde{T}(A)$ also satisfies $\psi|_B = \sigma$. Since $\psi(b) \leq \psi(a)$ whenever $b \preccurlyeq a$, we must have $\tau \leq \psi$. Since taking lower-semicontinuous regularisations is order-preserving, we have $\widetilde{\tau} \leq \psi$. By \cite[Proposition 3.2]{ERS11}, there exists $\varphi \in \widetilde{T}(A)$ such that $\psi =  \widetilde{\tau} + \varphi$. However, $\psi|_{B_+} = \widetilde{\tau}|_{B_+} = \sigma$ and so $\varphi$ vanishes on $B_+$. Since $A$ is simple, it follows that $\varphi = 0$ and so $\psi = \widetilde{\tau}$. Therefore, $\rho$ is injective.

Finally, we prove $\rho$ that is a homeomorphism. Fix $b \in \mathrm{Ped}(B) \setminus \lbrace 0 \rbrace$. Note that $b$ is also in $\mathrm{Ped}(A)$ and is full in both $A$ and $B$ by simplicity. Set $X_A := \lbrace \tau \in \widetilde{T}(A): \tau(b) = 1\rbrace$ and $X_B := \lbrace \tau \in \widetilde{T}(B): \tau(b) = 1\rbrace$. By \cite[Proposition 3.4]{TT15}, $X_A$ is a compact base for the cone $\widetilde{T}(A)$ and $X_B$ is a compact base for the cone $\widetilde{T}(B)$. Since $b \in B$, we have that $\rho(X_A) = X_B$. Hence, $\rho$ defines a continuous, affine bijection from $X_A$ to $X_B$. Since $X_A$ and $X_B$ are compact Hausdorff space, $\rho$ in fact defines an affine homeomorphism between compact bases for the the cones $\widetilde{T}(A)$ and $\widetilde{T}(B)$. Therefore, $\rho$ is a linear homeomorphism of the cones $\widetilde{T}(A)$ and $\widetilde{T}(B)$.  
\end{proof}

\subsection{Strict Comparison}\label{sec:CuntzComparison}

We first recall the definition of \emph{Cuntz sub-equivalence}. Let $A$ be a $\mathrm{C}^*$-algebra and $a,b \in A_+$. Then $a \precsim b$ if and only if there exists a sequence $(x_n)_{n\in\N}$ in $A$ such that 
\begin{equation}
\lim_{n\rightarrow \infty} \|x_n^*bx_n - a\| = 0.
\end{equation}
If $a \precsim b$ and $b \precsim a$, then $a$ is said to be \emph{Cuntz equivalent} to $b$. We shall write $[a]$ for the Cuntz equivalence class of $a$. 

The \emph{Cuntz semigroup} $\mathrm{Cu}(A)$ is the ordered abelian semigroup obtained by considering the Cuntz equivalence classes of positive elements in $A \otimes \K$ under orthogonal addition and the order induced by Cuntz subequivalence; see \cite{CEI08}. If one only considers the Cuntz equivalence classes of positive elements in $\bigcup_{k=1}^\infty M_k(A)$, then one obtains the \emph{classical Cuntz semigroup} $W(A)$; see \cite{To11}. 

Informally, a $\mathrm{C}^*$-algebra $A$ has \emph{strict comparison} if traces determine the Cuntz comparison theory. In order to formalise this notion, we need to recall the \emph{rank function} associated to a lower-semicontinuous quasitrace. Suppose $\tau:A_+ \rightarrow [0,\infty]$ is a lower-semicontinuous quasitrace. Then the rank function $d_\tau:(A \otimes \K)_+ \rightarrow [0,\infty]$ is given by
\begin{equation}
d_\tau(a) = \lim_{n\rightarrow\infty} \tau(a^{1/n}),
\end{equation} 
where we have made use of the unique extension of $\tau$ to $A\otimes\K$. We have $d_\tau(a) \leq d_\tau(b)$ whenever $a,b \in (A \otimes \K)_+$ satisfy $a \precsim b$ by \cite[Theorem II.2.2]{BH82}. Strict comparison can be viewed as a partial converse.

Since we will be adapting the methods of \cite{BBSTWW}, we shall be working with the same definition of strict comparison that is used there.
\begin{definition}[{\cite[Definition 1.5]{BBSTWW}}]
	\label{defn:StrictComp}
	A $\mathrm C^*$-algebra $A$ has \emph{strict comparison (of positive elements, with respect to bounded traces)} if 
	\begin{align}
	\label{eq:StrictComp}
	(\forall \tau\in T(A),\ d_\tau(a) < d_\tau(b))\Longrightarrow a\precsim b
	\end{align}
	for $k\in\mathbb{N}$ and $a,b\in M_k(A)_+$.
\end{definition}

We alert the reader to two facts about this definition. Firstly, it only concerns positive elements in $\bigcup_{k=1}^\infty M_k(A)$, so it is a property of the classical Cuntz semigroup $W(A)$. Secondly, we only require the condition $d_\tau(a) < d_\tau(b)$ to be shown when $\tau$ is a tracial state. 

In light of the potential confusion that could arise from the variety of definitions of strict comparison that appear in the literature, we include a proof that $\Z$-stability implies strict comparison in the sense of Definition \ref{defn:StrictComp} for the benefit of the reader. The key ingredient in the proof is that $W(A)$ is almost unperforated whenever $A$ is $\Z$-stable, which is due to R{\o}rdam \cite{Ro04}. 

\begin{proposition}\label{prop:ZstableToSC}
	Let $A$ be a simple, separable, $\Z$-stable $\mathrm{C}^*$-algebra with $Q\widetilde{T}(A) = \widetilde{T}_b(A) \neq 0$. Then $A$ has strict comparison of positive elements with respect to bounded traces.
\end{proposition}
\begin{proof}	
	As $A$ is $\Z$-stable, so is $A \otimes \K$. Hence, by \cite[Theorem 4.5]{Ro04}, $\mathrm{Cu}(A) = W(A \otimes \K)$ is almost unperforated. Applying \cite[Proposition 6.2]{ERS11} together with \cite[Proposition 4.4]{ERS11}, we find that $A$ has strict comparison in the following sense: for all $a,b \in (A\otimes\K)_+$ with $[a] \leq \infty[b]$ in $\mathrm{Cu}(A)$ if $d_\tau(a) < d_\tau(b)$ for all lower-semicontinuous quasitraces with $d_\tau(b) = 1$ then $[a] \leq [b]$ in $\mathrm{Cu}(A)$. 
	
	We show that under our hypothesis on $A$ this implies that $A$ has strict comparison in the sense of Definition \ref{defn:StrictComp}. Consider $a,b \in M_k(A)_+$ and let $\epsilon > 0$ and $f_\epsilon:[0,1]\rightarrow[0,1]$ be the function that is $0$ on $[0,\epsilon]$, affine on $[\epsilon,2\epsilon]$ and $1$ on $[2\epsilon,1]$. Since $M_k(A)$ is simple, there exists $n \in \mathbb{N}$ such that $[f_\epsilon(a)] \leq n[b] \leq \infty[b]$ in $\mathrm{Cu}(A)$ by \cite[Corollary II.5.2.12]{Bl06}. As $\epsilon$ is arbitrary, we have $[a] \leq \infty[b]$.
	
	Since $Q\widetilde{T}(A) = \widetilde{T}_b(A)$, if $d_\tau(a) < d_\tau(b)$ for all $\tau \in T(A)$, then $a \precsim b$ in $A\otimes \K$. As $M_k(A)$ is a hereditary subalgebra of $A\otimes\K$, we have $a \precsim b$ in $M_k(A)$ by \cite[Lemma 2.2(iii)]{KR00}. 
\end{proof}

\begin{remark}\label{rem:ZstableToSC}
	By replacing $M_k(A)$ with $A \otimes \K$ in the proof of Proposition \ref{prop:ZstableToSC}, we see that \eqref{eq:StrictComp} holds for all $a,b \in (A \otimes \K)_+$.  Therefore, $A$ also has strict comparison by traces in the sense of \cite[Definition 3.1]{NgR16} under the hypotheses of Proposition \ref{prop:ZstableToSC}.
\end{remark}

\subsection{Ultraproducts and Kirchberg's Epsilon Test}

Let $\omega$ be a free ultrafilter on $\mathbb{N}$, which we regard as fixed for the entirety of the paper. The \emph{ultraproduct} $\prod_{n\rightarrow\omega} A_n$ of a sequence of $\mathrm{C}^*$-algebras $(A_n)_{n\in\N}$ is defined by
\begin{equation}
\prod_{n\rightarrow\omega} A_n := \frac{\prod_{n\in\N} A_n}{\lbrace (a_n)_{n\in\N} \in \prod_{n\in\N} A_n: \lim\limits_{n\rightarrow\omega} \|a_n\| = 0\rbrace}.	
\end{equation}
The \emph{ultrapower} $A_\omega$ of a $\mathrm{C}^*$-algebra $A$ is the ultraproduct of the constant sequence $(A_n)_{n\in\N}$ with $A_n = A$ for all $n\in\N$. We identify $A$ with the subalgebra of $A_\omega$ given by constant sequences $(a)_{n\in\N}$.

Every sequence $(\tau_n)_{n\in\N}$ where $\tau_n \in T(A_n)$ defines a tracial state on the ultrapower $\prod_{n\rightarrow\omega} A_n$ via $(a_n) \mapsto \lim_{n\rightarrow\omega} \tau_n(a_n)$. Tracial states of this form are known as \emph{limit traces}. The set of all limit traces will be denoted by $T_\omega(\prod_{n\rightarrow\omega} A_n)$. 

Not all traces on an ultraproduct are limit traces but we have the following density result due to Ng--Robert \cite[Theorem 1.2]{NgR16} (generalising an earlier result of Ozawa \cite[Theorem 8]{Oz13}). 
\begin{theorem}[{\cite{NgR16,Oz13}}]\label{thm:NoSillyTraces}
	Let $(A_n)_{n\in\N}$ be a sequence of simple, separable, $\Z$-stable $\mathrm{C}^*$-algebras with $Q\widetilde{T}(A_n) = \widetilde{T}_b(A_n)$ for all $n \in \N$. Then $T_\omega(\prod_{n\rightarrow\omega} A_n)$ is weak$^*$-dense in $T(\prod_{n\rightarrow\omega} A_n)$.
\end{theorem}
\begin{proof}
	By Proposition \ref{prop:ZstableToSC} and Remark \ref{rem:ZstableToSC}, each $A_n$ has strict comparison by traces in the sense of \cite[Definition 3.1]{NgR16}. The result now follows by \cite[Theorem 1.2]{NgR16}.
\end{proof}

We shall also need uniform tracial ultraproducts. Recall that any trace $\tau \in T(A)$ defines a 2-seminorm $\|a\|_{2,\tau} := \tau(a^*a)^{1/2}$. The uniform 2-seminorm is then defined by
\begin{equation}
\|a\|_{2,T(A)} := \sup_{\tau \in T(A)} \|a\|_{2,\tau} = \sup_{\tau \in T(A)} \tau(a^*a)^{1/2}.
\end{equation}
We can then define the \textit{uniform tracial ultraproduct} of a sequence of $\mathrm{C}^*$-algebras $(A_n)_{n\in\N}$ by
\begin{equation}
\prod^{n\rightarrow\omega} A_n := \frac{\prod_{n\in\N} A_n}{\lbrace (a_n)_{n\in\N} \in \prod_{n\in\N} A_n: \lim\limits_{n\rightarrow\omega} \|a_n\|_{2,T(A_n)} = 0\rbrace}.	
\end{equation}
The \emph{uniform tracial ultrapower} $A^\omega$ of a $\mathrm{C}^*$-algebra $A$, which can be defined as the uniform tracial ultraproduct of the constant sequence $(A_n)_{n\in\N}$ with $A_n = A$ for all $n\in\N$, was introduced in \cite{CETWW}. We identify $A$ with the subalgebra of $A^\omega$ given by constant sequences $(a)_{n\in\N}$.

Since $\|a\|_{2,T(A)} \leq \|a\|$ for all $a \in A$, there exists a canonical surjection from the ultraproduct to the uniform tracial ultraproduct. The kernel of this $^*$-homomorphism is the \emph{trace kernel ideal} given by
\begin{align}
J_{(A_n)} &:= \lbrace (a_n)_{n\in\N} \in \prod_{n\rightarrow\omega} A_n: \lim_{n\rightarrow\omega} \|a_n\|_{2,T(A_n)} = 0\rbrace \notag \\
&= \lbrace x \in \prod_{n\rightarrow\omega} A_n: \|x\|_{2,\tau} = 0, \ \tau \in T_\omega(\prod_{n\rightarrow\omega}A_n) \rbrace.
\end{align}
It follows that limit traces also induce traces on the uniform tracial ultraproduct. In the ultrapower case, we therefore use a unified notation $T_\omega(A)$ for the limit traces on $A_\omega$ or the induced traces on $A^\omega$.

An important tool for working with ultrapowers are re-indexing arguments, which allow one to find elements of the ultrapower exactly satisfying some given condition provided one can find elements of the ultrapower which approximately satisfy the condition for any given tolerance. A precise and very general formulation of such re-indexing arguments is \textit{Kirchberg's Epsilon Test}, which we state below.

\begin{lemma}[{Kirchberg's Epsilon Test, \cite[Lemma A.1]{Kir06}}]\label{epstest}
	Let $X_1,X_2,\dots$ be a sequence of non-empty sets, and for each $k,n\in\N$, let $f^{(k)}_n:X_n\rightarrow [0,\infty)$ be a function.
	Define $f^{(k)}_\omega:\prod_{n=1}^\infty X_n\to[0,\infty]$ by $f^{(k)}_\omega((s_n)_{n=1}^\infty):=\lim_{n\rightarrow\omega}f^{(k)}_n(s_n)$ for $(s_n)\in\prod_{n=1}^\infty X_n$.  Suppose that for all $m\in\N$ and $\epsilon>0$, there exists $(s_n)_{n=1}^\infty\in \prod_{n=1}^\infty X_n$ with $f^{(k)}_\omega((s_n))<\epsilon$ for $k=1,\dots,m$.  Then there exists $(t_n)_{n=1}^\infty\in \prod_{n=1}^\infty X_n$ such that $f^{(k)}_\omega((t_n))=0$ for all $k\in\N$.
\end{lemma}

\subsection{Stable Rank One}

A unital $\mathrm{C}^*$-algebra $A$ is said to have \emph{stable rank one} if the invertible elements form a dense subset. In this paper, we shall make use of the following non-unital generalisation.

\begin{definition}\label{def:StableRankOneInUnitisation}
	Let $A$ be a $\mathrm{C}^*$-algebra. We say that $A$ has \emph{stable rank one in $A^\sim$} if every element of $A$ is a limit of invertible elements in $A^\sim$. 
\end{definition}

In the unital case, $A^\sim \cong A \oplus \mathbb{C}$, so $A$ has stable rank one in $A^\sim$ if and only if $A$ has stable rank one. In the non-unital case, $A$ having stable rank one in $A^\sim$ is weaker than requiring that $A^\sim$ itself has stable rank one; see \cite[Example 3.4]{Rob16}.

A related notion is Robert's \emph{almost stable rank one} \cite[Definition 3.1]{Rob16}, which requires that, for all hereditary subalgebras $B \subseteq A$, $B$ has stable rank one in $B^\sim$. Robert proved the following.

\begin{theorem}[{\cite[Corollary 3.2]{Rob16}}]\label{thm:ZstableProjectionless}
	Let $A$ be a $\mathcal{Z}$-stable, projectionless $\mathrm{C}^*$-algebra. Then $A$ has almost stable rank one. In particular, $A$ has stable rank one in $A^\sim$.
\end{theorem}

We now prove that having stable rank one in the unitisation passes to ultraproducts. We employ the notation $[(a_n)]$ for the element of the ultraproduct defined by the bounded sequence $(a_n)$. First, let us record that taking unitisations commutes with taking the ultraproduct. The proof of this lemma is straightforward and we omit it.

\begin{lemma}\label{lem:UnitisationUltraproduct}
	Let $(A_n)_{n\in\N}$ be a sequence of $\mathrm{C}^*$-algebras. The canonical inclusion $\prod_{n \rightarrow \omega} A_n \rightarrow \prod_{n \rightarrow \omega} A_n^\sim$ extends to an isomorphism 
	\begin{equation}
	\left(\prod_{n \rightarrow \omega} A_n\right)^\sim \cong \prod_{n \rightarrow \omega} A_n^\sim.
	\end{equation}
\end{lemma}
%

We now proceed to show that having stable rank one in the unitisation passes to ultraproducts. 

\begin{proposition}\label{prop:SR1inUnitisationUltrapowers}
	Let $(A_n)$ be a sequence of $\mathrm{C}^*$-algebras. Suppose for each $n \in \mathbb{N}$, $A_n$ has stable rank one in $A_n^\sim$. Then $A_\omega := \prod_{n \rightarrow \omega} A_n$ has stable rank one in $A_\omega^\sim$. 
\end{proposition}
\begin{proof}
	Let $x \in A_\omega$ and say $x = [(a_n)]$. By Theorem \ref{thm:ZstableProjectionless} and \cite[Lemma 1.20]{BBSTWW}, 
	for each $n \in \mathbb{N}$ there is a unitary $u_n \in A_n^\sim$ such that $a_n \approx_{1/n} u_n|a_n|$. We then have $x = [(u_n)]|x| \in \prod_{n \rightarrow \omega} A_n^\sim$. By \cite[Lemma 1.20]{BBSTWW} 
	once more, $x$ is a norm limit of invertible elements in $\prod_{n \rightarrow \omega} A_n^\sim$. By  Lemma \ref{lem:UnitisationUltraproduct}, $\prod_{n \rightarrow \omega} A_n^\sim$ is just $A_\omega^\sim$. 
\end{proof}

\subsection{Complemented Partitions of Unity}
The key technical tool in \cite{CETWW} was the complemented partitions of unity technique which enabled Theorem \ref{thm:NewMain2} to be proven in the unital case.
This property is best formulated in terms of the tracial ultrapower $A^\omega$ of a separable $\mathrm{C}^*$-algebra with $T(A)$ non-empty and compact. These assumptions imply that $A^\omega$ is unital, with any sequential approximate identity representing the unit \cite[Proposition 1.11]{CETWW}. 
We refer to \cite[Definition G]{CETWW} for a detailed explanation of the ideas behind this definition. 

\begin{definition}\label{defn:CPOU}
	Let $A$ be a separable $\mathrm{C}^*$-algebra with $Q\widetilde{T}(A) = \widetilde{T}_b(A) \neq 0$ and $T(A)$ compact.
	We say that $A$ has \emph{complemented partitions of unity} (CPoU) if for every $\|\cdot\|_{2,T_\omega(A)}$-separable subset $S$ of $A^\omega$, every family $a_1,\dots,a_k \in (A^\omega)_+$, and any scalar
	\begin{equation}
	\label{eq:CPoUTraceIneq1}
	\delta>\sup_{\tau\in T_\omega(A)}\min_{i=1,\dots,k}\tau(a_i),
	\end{equation}
	there exist orthogonal projections $p_1,\dots,p_k\in A^\omega\cap S'$ such that
	\begin{equation}
	p_1+\cdots+p_k = 1_{A^\omega} \ \text{and}\ 
	\tau(a_ip_i)\leq \delta\tau(p_i), \tau\in T_\omega(A), i=1,\dots,k.
	\end{equation}
\end{definition}

The following theorem gives sufficient conditions for a $\mathrm{C}^*$-algebra to have complemented partitions of unity. Although not necessary for our purposes, the hypothesis of $\Z$-stability can be weakened to \emph{uniform property $\Gamma$}; see \cite[Section 2]{CETWW} for more details.

\begin{theorem}[{\cite[Theorem I]{CETWW}\label{thm:CPoU}}]
	Let $A$ be a separable, nuclear, $\Z$-stable $\mathrm{C}^*$-algebra with $Q\widetilde{T}(A) = \widetilde{T}_b(A) \neq 0$ and $T(A)$ compact. Then $A$ has complemented partitions of unity.
\end{theorem}

\section{Reductions}\label{sec:Reductions}

In this section, we show how Brown's Theorem \cite[Theorem 2.8]{Br77} can be used to reduce the task of proving Theorem \ref{thm:NewMain2} in general to proving it for unital $\mathrm{C}^*$-algebras and for stably projectionless $\mathrm{C}^*$-algebras with a compact trace space.
We begin with the general statement of Brown's Theorem.

\begin{theorem}[{\cite[Theorem 2.8]{Br77}}]\label{thm:Brown} 
	Let $B$ be a full hereditary subalgebra of a $\mathrm{C}^*$-algebra $A$. Suppose both $A$ and $B$ are $\sigma$-unital. Then $B$ is stably isomorphic to $A$.
\end{theorem}
In our applications, we shall be working with $\mathrm{C}^*$-algebras that are simple and separable. Hence, the fullness and $\sigma$-unitality conditions will be satisfied. We shall therefore use the following form of Brown's Theorem.
\begin{corollary} \label{cor:Brown} 
	Let $B$ be a non-zero hereditary subalgebra of a simple, separable $\mathrm{C}^*$-algebra $A$. Then $B$ is stably isomorphic to $A$.
\end{corollary}

The utility of Brown's Theorem for this paper derives from the fact that the hypotheses and conclusion of Theorem \ref{thm:NewMain2} are invariant under stable isomorphism. We state this formally below.  
\begin{proposition}[{\cite{KW04,WZ10,TW07}}]
		\label{thm:StableIsoInvariants}
	Let $A$ be a $\mathrm{C}^*$-algebra. Then
	\begin{enumerate}
		\item $A$ is simple if and only if $A \otimes \mathbb{K}$ is simple,
		\item $A$ is separable if and only if $A \otimes \mathbb{K}$ is separable,
		\item $A$ is nuclear if and only if $A \otimes \mathbb{K}$ is nuclear,
		\item $\dr(A) = \dr(A \otimes \mathbb{K})$,
		\item $\mathrm{dim}_{\mathrm{nuc}}(A) = \mathrm{dim}_{\mathrm{nuc}}(A  \otimes \mathbb{K})$,
		\item $A$ is separable and $\Z$-stable if and only if $A \otimes \mathbb{K}$ is separable and $\Z$-stable.
	\end{enumerate}
\end{proposition}
\begin{proof}
	Properties (i-iii) are well known; see for example \cite[Chapter IV.3]{Bl06}. Part (iv) is \cite[Corollary 3.9]{KW04}. Part (v) is \cite[Corollary 2.8]{WZ10}. Part (vi) is \cite[Corollary 3.2]{TW07}. 
\end{proof}

Next, we recall that a $\mathrm{C}^*$-algebra $A$ is \emph{stably projectionless} if there are no non-zero projections in $A \otimes \mathbb{K}$. By definition, this property is preserved under stable isomorphism. Stably projectionless $\mathrm{C}^*$-algebras can be viewed as highly non-unital $\mathrm{C}^*$-algebras. Indeed, the following folklore result establishes a dichotomy for simple, separable $\mathrm{C}^*$-algebras.  
\begin{proposition}\label{prop:Dichotomy}
	Let $A$ be a non-zero, simple, separable $\mathrm{C}^*$-algebra. Then exactly one of the following holds.
	\begin{itemize}
		\item[(a)] $A$ is stably isomorphic to a unital $\mathrm{C}^*$-algebra.
		\item[(b)] $A$ is stably projectionless. 	
	\end{itemize}	 
\end{proposition}
\begin{proof}
	Let $A$ be a simple, separable $\mathrm{C}^*$-algebra $A$. Then $A \otimes \K$ is simple and separable by Proposition \ref{thm:StableIsoInvariants}. 
	Suppose that $A$ is not stably projectionless. Then there exists a non-zero projection $p \in A \otimes K$. Set $B := p(A \otimes \K)p$. Then $B$ is a unital $\mathrm{C}^*$-algebra with unit $1_B = p$. Moreover, $B$ is a non-zero hereditary subalgebra of $A \otimes \K$.  Therefore, $B$ is stably isomorphic to $A \otimes \K$ by Corollary \ref{cor:Brown}, and hence is stably isomorphic to $A$. 
	
	Now suppose that $A$ is stably isomorphic to a unital $\mathrm{C}^*$-algebra $B$. Then there exists an isomorphism $\phi:B \otimes \K \rightarrow A \otimes \K$. Writing $1_B$ for the unit of $B$ and $e_{ij}$ for the matrix units of $\K$, we have that $\phi(1_B \otimes e_{ii})$ is non-zero projection in $A \otimes \K$. Hence, $A$ cannot be stably projectionless. 
\end{proof}

This dichotomy justifies the terminology \emph{stably unital} for the non--stably projectionless, simple, separable $\mathrm{C}^*$-algebras. The stably unital case of Theorem \ref{thm:NewMain2} follows immediately from \cite[Theorem B]{CETWW} together with Propositions \ref{prop:Dichotomy} and \ref{thm:StableIsoInvariants}. The stably projectionless case on the other hand requires a further reduction and a technical modification of the methods of \cite{BBSTWW}. The purpose of the additional reduction is to pass to the case where the trace space is compact, and it is based on the following folklore result.

\begin{lemma}\label{lem:HerditartyContRankFunction}
Let $A$ be a simple, separable $\mathrm{C}^*$-algebra with $\widetilde{T}(A) \neq  0$.  Let $A_0 := \overline{a(A \otimes \K)a}$ be the hereditary subalgebra generated by a non-zero positive contraction $a \in (A \otimes \K)_{+,1}$ for which the function $\tau \mapsto d_\tau(a)$ is continuous and finite-valued. Then $\widetilde{T}(A_0) = \widetilde{T}_b(A_0) \neq  0$ and $T(A_0)$ is compact.
\end{lemma} 
\begin{proof}

By Proposition \ref{prop:ExtendingTraces}, the restriction map $\rho:\widetilde{T}(A \otimes \K) \rightarrow \widetilde{T}(A_0)$ is a linear homeomorphism. Let $\sigma \in \widetilde{T}(A_0)$. Then $\sigma$ has an extension $\tau := \rho^{-1}(\sigma) \in \widetilde{T}(A \otimes \K)$. By \cite[Proposition 2.4]{Ti14}, we have $\|\sigma\|_{A_0^*} = d_{\tau}(a) < \infty$. Therefore, $\widetilde{T}(A_0) = \widetilde{T}_b(A_0) \neq  0$. Since $\sigma \mapsto d_{\rho^{-1}(\sigma)}(a)$ is continuous, $T(A_0)$ is a weak$^*$-closed subspace of the unit ball of $A_0^*$. Therefore, $T(A_0)$ is compact. 
\end{proof}

We now explain how the results of \cite{ERS11} can be used to prove the existence of positive contractions with continuous rank functions under suitable hypotheses.

\begin{proposition}	\label{prop:ExistenceRankFunction}
	Let $A$ be a simple, separable, $\Z$-stable $\mathrm{C}^*$-algebra with $Q\widetilde{T}(A) = \widetilde{T}(A) \neq 0$. Let $f:\widetilde{T}(A) \rightarrow [0,\infty)$ be a strictly positive, continuous, linear function. Then there exists a non-zero positive contraction $a \in (A \otimes \K)_{+,1}$ with $d_\tau(a) = f(\tau)$ for all $\tau \in \widetilde{T}(A)$. 
\end{proposition}
\begin{proof}
	Following \cite[Section 4.1]{ERS11}, we write $F(\mathrm{Cu(A)})$ for the space of functionals on the Cuntz semigroup of $A$. In \cite[Theorem 4.4]{ERS11}, it is shown that all functionals on the Cuntz semigroup are of the form $d_\tau$ for some lower-semicontinuous quasitrace $\tau$ on $A$. However, we alert the reader to the fact that quasitraces are not assumed to be densely defined in \cite{ERS11}. Since $A$ is simple, this means that either $\tau \in Q\widetilde{T}(A)$ or $\tau$ is the \emph{trivial quasitrace}, which satisfies $\tau(0) = 0$ and is infinite otherwise. 
	
We now consider the topology on $F(\mathrm{Cu(A)})$, defined in general in \cite[Section 4.1]{ERS11}, and its relation to the topology on $\widetilde{T}(A)$, which is given by pointwise convergence on $\mathrm{Ped}(A)$. By \cite[Theorem 4.4]{ERS11} and our assumption that all quasitraces are additive, the topology on $F(\mathrm{Cu(A)})$ agrees with the topology on set $\widehat{T}(A)$ of (not-necessarily-densely-defined) lower-semicontinuous, additive quasitraces defined in \cite[Section 3.2]{ERS11}.\footnote{In \cite{ERS11}, the notation $T(A)$ is used instead of $\widehat{T}(A)$, but this clashes with the notation for the tracial states used in this paper.} This topology is shown to be compact and Hausdorff in \cite[Theorem 3.7]{ERS11}. By \cite[Theorem 3.10]{ERS11}, the restriction of this topology to $\widetilde{T}(A)$ is pointwise convergence on $\mathrm{Ped}(A)$. Since $\widehat{T}(A) \setminus \widetilde{T}(A)$ is just one point, it follows that the topology on $\widehat{T}(A)$ is simply the one point compactification of the topology on $\widetilde{T}(A)$.  

By \cite[Proposition 3.4]{TT15}, the cone $\widetilde{T}(A)$ has a compact base $K$. Since $f$ is strictly positive and continuous, $\inf_{\tau \in K} f(\tau) > 0$. Hence, we may extend $f$ to the one-point compactification of $\widetilde{T}(A)$ by setting $f(\infty) = \infty$ and the resulting map is still continuous. It follows that $f$ defines an element of the dual cone $L(F(\mathrm{Cu(A)}))$; see \cite[Section 5.1]{ERS11}. Therefore, as $A$ is $\Z$-stable, there exists $a \in (A \otimes \K)_{+,1}$ such that $f(\tau) = d_\tau(a)$ for all $\tau \in \widetilde{T}(A)$ by \cite[Theorem 6.6]{ERS11}.       
\end{proof}

We end this section with the following summary of all the reductions.
\begin{theorem}\label{thm:simple.reduction}
	Let $A$ be a non-zero, simple, separable, exact, $\Z$-stable $\mathrm{C}^*$-algebra. Then one of the following holds.
	\begin{itemize}
		\item[(a)]	$A$ is stably isomorphic to a unital $\mathrm{C}^*$-algebra,
		\item[(b)]	$A$ is stably isomorphic to a stably projectionless $\mathrm{C}^*$-algebra $A_0$ with $Q\widetilde{T}(A_0) = \widetilde{T}_b(A_0) \neq 0$ and $T(A_0)$ is compact.   
	\end{itemize}
\end{theorem}
\begin{proof}
	Suppose (a) does not hold. Then $A$ is stably projectionless by Proposition \ref{prop:Dichotomy}. By Theorem \ref{thm:TracialStatesExist}, $Q\widetilde{T}(A) \neq 0$. Since $A$ is exact, the non-unital version of Haagerup's Theorem gives $Q\widetilde{T}(A) = \widetilde{T}(A) \neq 0$; see \cite{Ha14} and \cite[Remark 2.29(i)]{BK04}. Since $\widetilde{T}(A)$ is a cone with a compact base, there exists a strictly positive, continuous, linear function $f:\widetilde{T}(A) \rightarrow [0,\infty)$. By Proposition \ref{prop:ExistenceRankFunction}, there is a positive contraction $a \in (A \otimes \K)_{+,1}$ such that $f(\tau) = d_\tau(a)$ for all $\tau \in \widetilde{T}(A)$. Set $A_0 := \overline{a(A \otimes \K)a}$. By Lemma \ref{lem:HerditartyContRankFunction}, $Q\widetilde{T}(A_0) = \widetilde{T}_b(A_0) \neq 0$ and $T(A_0)$ compact. By Corollary \ref{cor:Brown}, $A$ is stably isomorphic to $A_0$. Hence, $A_0$ is stably projectionless.
\end{proof}

\section{Existence}

Let $A$ be a separable, nuclear C$^*$-algebra with complemented partitions of unity (CPoU). In this section, we will construct a sequence of maps $A \xrightarrow{\theta_n} F_n \xrightarrow{\eta_n} A$, where $F_n$ are finite dimensional $\mathrm{C}^*$-algebras, $\theta_n$ are c.p.c.\ maps and $\eta_n$ are c.p.c.\ order zero maps, which induces a $^*$-homomorphism $A \rightarrow A^\omega$ that agrees with the diagonal inclusion $A \rightarrow A^\omega$.

We will do this in two steps. First, we will fix a trace $\tau$ and produce maps $A \to F \to A$ that approximate the identity map on $A$ in $\|\cdot \|_{2, \tau}$-norm. We shall then construct the required sequence of maps using complemented partitions of unity (CPoU).

The following lemma will be deduced from \cite[Lemma 5.1]{CETWW}, but it can also be proved by directly applying the methods of \cite[Lemma 2.5]{BCW16}.

\begin{lemma}\label{lem:preexistence}
	Let $A$ be a separable, nuclear $\mathrm{C}^*$-algebra and let $\tau \in T(A)$. For any finite subset $\mathcal{F} \subseteq A$ and $\epsilon >0$ there exist a finite dimensional $\mathrm{C}^*$-algebra $F$, a c.p.c.\ map $\theta: A \to F$, and a c.p.c.\ order zero map $\eta: F \to A$ such that
	\begin{align}
	\label{eq:OneTraceMap1}
	\|\theta(a)\theta(b)\| &< \epsilon \quad && \text{for }a,b \in \mathcal{F} \text{ such that }ab=0\text{, and} \\ 
	\label{eq:OneTraceMap2}
	\|\eta \circ \theta(a)-a\|_{2,\tau} &< \epsilon &&\text{for }a \in \mathcal{F}.
	\end{align}
	If all traces are quasidiagonal,\footnote{See Definition \ref{dfn:QDtraces}.} one can additionally arrange that
	\begin{align}
	\| \theta(ab) - \theta(a)\theta(b) \| < \epsilon, && a,b \in \mathcal{F}.
	\end{align}
\end{lemma}

\begin{proof}
	The trace $\tau$ extends to a trace on $A^\sim$. By \cite[Lemma 5.1]{CETWW} applied to $A^\sim$, there exist a finite dimensional $F$, a c.p.c.\ map $\tilde \theta: A^\sim \to F$, and a c.p.c.\ order zero map $\tilde \eta : F \to A^\sim$ such that
	\begin{align}
	\|\tilde\theta(a)\tilde\theta(b)\| &< \frac{\epsilon}{2} \quad && \text{for }a,b \in \mathcal{F} \text{ satisfying }ab=0\text{, and} \\ 
	\|\tilde\eta \circ \tilde\theta(a)-a\|_{2,\tau} &< \frac{\epsilon}{2} &&\text{for }a \in \mathcal{F}.
	\end{align}
	Let $(e_n)_{n\in \mathbb{N}}$ be an approximate identity of $A$. Then $e_n \nearrow 1_{A^\sim}$ in $\|\cdot\|_{2,\tau}$. Hence, the c.p.c.\ maps $\hat\eta_n: F \to A$ given by $\hat\eta_n(x)= e_n \tilde\eta(x)e_n$ converge to $\tilde \eta$ in the point-$\|\cdot\|_{2,\tau}$ topology. The sequence $\hat\eta_n$ is asymptotically order zero in $\|\cdot\|_{2,\tau}$. Since $F$ is finite dimensional, we can make use of order zero lifting to obtain a sequence of c.p.c.\ order zero maps $\eta_n:F \rightarrow A$ converging to $\tilde \eta$ in the point-$\|\cdot\|_{2,\tau}$ topology.\footnote{Indeed, let $J_\tau := \lbrace (a_n)_{n \in\N} \in \ell^\infty(A): \lim_{n\rightarrow\infty}\tau(a_n^*a_n) = 0\rbrace$ and consider the diagonal map $(\hat\eta_n):F \rightarrow \ell^\infty(A)/J_\tau$. This map is c.p.c.\ order zero and so has a c.p.c.\ order zero lift $(\eta_n): F \rightarrow \ell^\infty(A)$ by  \cite[Proposition 1.2.4]{Wi09}.}

	Set $\theta := \tilde \theta |_A$. Choose $n \in \N$ such that $\|\eta_n(\theta(a)) - \tilde{\eta}(\theta(a))\|_{2,\tau} < \tfrac{\epsilon}{2}$ for all $a \in \mathcal{F}$. We then have
	\begin{align}
	\|\eta \circ \theta(a)  - a \|_{2, \tau} & < \| \tilde{\eta} \circ \theta (a) - a  \|_{2,\tau} + \frac{\epsilon}{2} < \epsilon, && a \in \mathcal{F}.
	\end{align}
	If all traces are quasidiagonal, the map $\tilde \theta$ given by \cite[Lemma 5.1]{CETWW} is approximately a $^*$-homomorphism. Hence, so is $\theta$.
\end{proof}

With the previous lemma in hand, we can now utilise complemented partitions of unity (CPoU) to prove the following.

\begin{lemma}
	\label{lem:NiceFactoring}
	Let $A$ be a separable, nuclear $\mathrm{C}^*$-algebra with $Q\widetilde{T}(A) = \widetilde{T}_b(A) \neq 0$ and $T(A)$ compact. Suppose A has CPoU.
	Then there exists a sequence of c.p.c.\ maps $\phi_n:A \to A$ which factor through finite dimensional algebras $F_n$ as
	\begin{equation}
	\label{eq:NiceFactoring1}
	\xymatrix{A\ar[dr]_{\theta_n}\ar[rr]^{\phi_n}&&A\\&F_n \ar[ur]_{\eta_n}}
	\end{equation}
	with $\theta_n$ c.p.c.\ and $\eta_n$ c.p.c.\ order zero, in such a way that the induced map $(\theta_n)_{n=1}^\infty:A \to \prod_\omega F_n$ is order zero, and the induced map $\overline{\Phi}=(\phi_n)_{n=1}^\infty:A\rightarrow A^\omega$ agrees with the diagonal inclusion $A \to A^\omega$.
	
	If all traces on $A$ are quasidiagonal, then we may arrange that the induced map $(\theta_n)_{n=1}^\infty:A \to \prod_\omega F_n$ is a $^*$-homomorphism.
\end{lemma}

\begin{proof}
	As in \cite[Lemma 5.2]{CETWW}, by a standard application of Kirchberg's Epsilon Test, it suffices to show that for a finite set $\mathcal F \subseteq A$ and a tolerance $\epsilon>0$, there is a sequence $(F_n, \theta_n, \eta_n)$ such that $\theta_n: A \to F_n$ is approximately order zero (or approximately multiplicative if all traces are quasidiagonal), $\eta_n: F_n \to A$ is an order zero map, and the induced map $\overline{\Phi}_\epsilon= (\eta_n \circ \theta_n)_{n=1}^\infty: A \to A^\omega$ satisfies $\| a - \overline{\Phi}_\epsilon (a) \|_{2, T_\omega (A)} < \epsilon$ for all $a \in \mathcal{F}$. In fact, we will arrange for all the $F_n$ to be the same finite dimensional algebra $F$, and all the $\theta_n$ to be the same map $\theta$.
	
	Let $\mathcal{F} \subseteq A$ be a finite subset and $\epsilon> 0$.
	By Lemma \ref{lem:preexistence}, for any $\tau \in T(A)$ there exist a finite dimensional $\mathrm{C}^*$-algebra $F_\tau$, a c.p.c.\ map $\theta_\tau: A \to F_\tau$, and an order zero map $\eta_\tau: F_\tau \to A$ such that
	\begin{align}
	\| \theta(a) \theta(b) \| &< \epsilon, && a,b \in \mathcal{F} \text{ such that } ab = 0,\\
	\| \eta_\tau \circ \theta_\tau(x) - x \|_{2, \tau}^2 &< \frac{\epsilon^2}{|\mathcal{F}|}, && x \in \mathcal{F}.
	\end{align}
	
	Set $a_\tau := \sum\limits_{x \in \mathcal{F}} | x - \eta_\tau\tau \circ \theta_\tau (a) |^2$.
	By compactness of $T(A)$, there exist $\tau_1, \ldots, \tau_k \in T(A)$ such that for all $\tau \in T(A)$ there is some $\tau_i$ such that $\tau(a_{\tau_i}) < \epsilon^2$. 
	
	By CPoU, there exist pairwise orthogonal projections $p_1, \ldots, p_k \in A^\omega \cap A'$ adding up to $1_{A^\omega}$ such that $\tau(a_i p_i ) \leq \epsilon^2 \tau(p_i)$ for all $\tau \in T_\omega (A)$.
	Set $F := \bigoplus_{i=1}^k F_{\tau_i}$, $\theta: A \to F$ and $\eta: F \to A^\omega$ by
	\begin{equation}
	\theta(a) := ( \theta_{\tau_1}(a), \ldots, \theta_{\tau_k}(a)), \quad \eta(x_1, \ldots, x_k) := \sum_{i=1}^k \eta_{\tau_i}(x_i) p_i, 
	\end{equation}
	where $a \in A$ and $x_i \in F_{\tau_i}$. By construction (see \cite[Lemma 5.2. Equation (5.16)]{CETWW}), we obtain
	\begin{equation}
	\| a - \eta \circ \theta (a)  \|_{2, T_\omega(A)} < \epsilon, \qquad a \in \mathcal{F}.
	\end{equation}
	
	By \cite[Proposition 1.2.4]{Wi09}, $\eta: F \to A^\omega$ can be lifted to a sequence of order zero maps $\eta_n: F \to A$. Thus the sequence $(F, \theta, \eta_n)$ is the required sequence.
	
	Finally, if all traces are quasidiagonal, the map $\theta$ is approximately multiplicative by Lemma \ref{lem:preexistence}. Combining the previous argument with Kirchberg's Epsilon Test yields that the induced map $(\theta_n): A \to \prod_\omega F_n$ is a $^*$-homomorphism.
\end{proof}

\section{Unitisation}

In this section, we prove that a c.p.c.\ order zero map $\phi:A \rightarrow B_\omega$ from a separable $\mathrm{C}^*$-algebra into a $\mathrm{C}^*$-ultrapower extends to a c.p.c.\ order zero map $\phi^\sim:A^\sim \rightarrow B_\omega$. Moreover, under appropriate conditions, Dini's Theorem can be used to construct an extension for which the tracial behaviour of $\phi^\sim(1_{A^\sim})$ is determined by $\phi$. These results were inspired by the structure theory for order zero maps developed in \cite{WZ09} and the existence of \emph{supporting order zero maps} proved in \cite[Lemma 1.14]{BBSTWW}.
We begin with a technical lemma.    

\begin{lemma}\label{lem:unitising-orderzero}
	Let $\phi:A \rightarrow B$ be a c.p.c.\ order zero map between $\mathrm{C}^*$-algebras. Suppose that $h \in B$ is a positive contraction such that 
	\begin{equation}\label{eqn:OrderZeroIndentity1}
	\phi(a)\phi(b) = h\phi(ab), \qquad \qquad a,b \in A_+.
	\end{equation}
	Then the map $\phi^\sim:A^\sim \rightarrow B$ defined by $\phi^\sim(a+\lambda 1_{A^\sim}) := \phi(a) + \lambda h$ is c.p.c.\ order zero. 
\end{lemma}
\begin{proof}
	By \cite[Corollary 4.1]{WZ09}, there exists a $^*$-homomorphism $\pi:C_0(0,1] \otimes A \rightarrow B_\omega$ such that $\phi(a) = \pi(t \otimes a)$ for all $a \in A$, where $t$ denotes the canonical generator of the cone. In terms of $\pi$, equation (\ref{eqn:OrderZeroIndentity1}) gives $h\pi(t \otimes ab) = \pi(t^2 \otimes ab)$ for $a,b \in A_+$, from which we deduce that
	\begin{equation}
	h\pi(t \otimes a) = \pi(t^2 \otimes a), \qquad \qquad a \in A, \label{eqn:hAction1}
	\end{equation}
	since $(A_+)^2 = A_+$ and $A_+$ spans $A$. It then follows that $h^n\pi(t^m \otimes a) = \pi(t^{n+m} \otimes a)$ for $a \in A$ and for all $n,m \in \N_{\geq 1}$, from which we obtain
	\begin{equation}
	g(h)\pi(f \otimes a) =  \pi(gf \otimes a), \qquad \qquad a \in A, f,g \in C_0(0,1]\label{eqn:hAction3},	
	\end{equation}
	since $\mathrm{span}\lbrace t^n:n \in \N_{\geq 1}\rbrace$ is dense in $C_0(0,1]$. Taking adjoints, we also obtain $\pi(f \otimes a)g(h) =  \pi(fg \otimes a)$ for all $a \in A, f,g \in C_0(0,1]$.
	
	We now define a map $\pi^\sim:C_0(0,1] \odot A^\sim \rightarrow  B$ from the algebraic tensor product by setting $\pi^\sim(f \otimes (a + \lambda 1_{A^\sim})) := \pi(f \otimes a) + \lambda f(h)$ on elementary tensors. A straightforward computation using \eqref{eqn:hAction3} and its adjoint shows that $\pi^\sim$ is a $^*$-homomorphism. Hence, $\pi^\sim$ extends to a map $C_0(0,1] \otimes A^\sim \rightarrow  B$. Finally, define $\phi^\sim:A^\sim \rightarrow B$ by $\phi^\sim(x) := \pi^\sim(t \otimes x)$. Then $\phi^\sim$ is a c.p.c.\ order zero map and $\phi^\sim(a + \lambda 1_{A^\sim}) = \phi(a) + \lambda h$ as required. 
\end{proof}

We now prove the unitisation lemma for order zero maps. 
\begin{lemma}\label{lem:Unitisation}
	Let $A$, $B$ be $\mathrm{C}^*$-algebras with $A$ separable and let $\phi:A \rightarrow B_\omega$ be a c.p.c.\ order zero map. 
	\begin{itemize}	
		\item[(a)] There exists a c.p.c.\ order zero map $\phi^\sim:A^\sim \rightarrow B_\omega$ which extends $\phi$.
		\item[(b)] Suppose now that $T(B)$ is compact and non-empty. Let $(e_n)_{n\in\mathbb{N}}$ be an approximate unit for $A$ and suppose the function
		\begin{align*}
		\theta:\overline{T_\omega(B)}^{w*} &\rightarrow [0,1]\\
		\tau  &\mapsto \lim_{n \rightarrow \infty} \tau(\phi(e_n))
		\end{align*}
		is continuous. Then there exists a c.p.c.\ order zero map $\phi^\sim:A^\sim \rightarrow B_\omega$ which extends $\phi$ and satisfies $\tau(\phi^\sim(1_{A^\sim})) = \theta(\tau)$ for all $\tau \in \overline{T_\omega(B)}^{w*}$.
	\end{itemize}
\end{lemma}	
\begin{proof}
	(a) Let $(e_n)_{n\in\mathbb{N}}$ be an approximate unit for $A$. By \cite[Corollary 4.1]{WZ09}, there exists a $^*$-homomorphism $\pi:C_0(0,1] \otimes A \rightarrow B_\omega$ such that $\phi(a) = \pi(t \otimes a)$ for all $a \in A$, where $t$ denotes the canonical generator of the cone. For any $a, b \in A_+$, we have 
	\begin{align}
	\lim_{n\rightarrow\infty} \phi(e_n)\phi(ab) &= \lim_{n\rightarrow\infty} \pi(t^2 \otimes e_nab) \notag \\
	&= \pi(t^2 \otimes ab) \notag \\
	&= \phi(a)\phi(b). \label{eqn:CorrectInTheLimit}
	\end{align}
	
	We shall now prove the existence of a positive contraction $h \in B_\omega$ such that (\ref{eqn:OrderZeroIndentity1}) holds for all $a,b \in A_+$ by an application of Kirchberg's Epsilon Test (Lemma \ref{epstest}).  
	
	Let $X_n := B_{+,1}$ for all $n \in \mathbb{N}$. Let $\phi_n:A_+\rightarrow B_+$ be a sequence of functions such that $(\phi_n(a))_{n \in \N}$ is a representative for $\phi(a)$ for all $a \in A_+$.
	Fix a dense sequence $(a_n)_{n\in \mathbb{N}}$ in $A_+$. Define $f^{(r,s)}_n:X_n \rightarrow [0,\infty]$ for $r,s \in \mathbb{N}$ by
	\begin{equation}
	f^{(r,s)}_n(x) := \|x\phi_n(a_ra_s) - \phi_n(a_r)\phi_n(a_s)\|.
	\end{equation} 
	Then define $f^{(r,s)}_\omega:\prod_{n=1}^\infty  X_n \rightarrow [0,\infty]$ by $(x_n)_{n\in\mathbb{N}} \mapsto \lim_{n\rightarrow\omega} f^{(r,s)}_n(x_n).$
	
	Let $m \in \mathbb{N}$ and $\epsilon > 0$. By (\ref{eqn:CorrectInTheLimit}), there is $k \in \mathbb{N}$ such that
	\begin{align}
	\|\phi(e_k)\phi(a_ra_s) - \phi(a_r)\phi(a_s)\| &< \epsilon, &1 \leq r,s \leq m. 
	\end{align}
	
	Let $x = (x_n)_{n \in \mathbb{N}}$ be a sequence of positive contractions in $B$ representing $\phi(e_k)$.  Then $f^{(r,s)}_\omega(x) < \epsilon$ whenever $1 \leq r,s \leq m$. By Kirchberg's Epsilon Test, there exists a sequence of positive contractions $y = (y_n)_{n\in\mathbb{N}}$ in $B$ such that $f^{(r,s)}_\omega(y) = 0$ for all $r,s \in \mathbb{N}$. Let $h$ be the positive contraction in $B_\omega$ represented by $(y_n)_{n\in\mathbb{N}}$. Then $h$ satisfies (\ref{eqn:OrderZeroIndentity1}) for all $a,b \in \lbrace a_n: n \in \mathbb{N} \rbrace$. By density, $h$ satisfies (\ref{eqn:OrderZeroIndentity1}) for all $a,b \in A_+$. The result now follows by Lemma \ref{lem:unitising-orderzero}.

	(b) By Dini's Theorem, $\tau(\phi(e_n)) \nearrow \theta(\tau)$ uniformly for $\tau \in \overline{T_\omega(B)}^{w*}$.\footnote{Our convention is that approximate units for C$^*$-algebras are by default
	assumed to be increasing.} For each $l \in \mathbb{N}$, set 
	\begin{equation}\label{lem:unitisation.gamma.eq}
	\gamma_l := \max_{\tau \in \overline{T_\omega(B)}^{w*}} (\theta(\tau) - \tau(\phi(e_l))).
	\end{equation}
	Then $\gamma_l \geq 0$ as $\tau(\phi(e_n))$ increases with $n$, and $\lim_{l\rightarrow\infty} \gamma_l = 0$ as the convergence is uniform. 
	
	We shall now prove the existence of a positive contraction $h \in B_\omega$ such that (\ref{eqn:OrderZeroIndentity1}) holds for all $a,b \in A_+$ and that		
	\begin{equation}
	\tau(h) = \lim_{n \rightarrow \infty} \tau(\phi(e_n)), \qquad \tau \in \overline{T_\omega(B)}^{w*}. \label{eqn:hTrace}
	\end{equation} 
	Once again, we use Kirchberg's Epsilon Test (Lemma \ref{epstest}). 
	
	Let $X_n$, $\phi_n$, $f^{(r,s)}_n$, and $f^{(r,s)}_\omega$ be as in (a). Define $g^{(l,+)}_n,g^{(l,-)}_n:X_n \rightarrow [0,\infty]$ for $l \in \mathbb{N}$ by
	\begin{align}
	g^{(l,+)}_n(x) &:= \max\left(\sup_{\tau \in T(B)} (\tau(x) - \tau(\phi_n(e_l))) - \gamma_l, 0 \right),\\
	g^{(l,-)}_n(x) &:= \max\left(\sup_{\tau \in T(B)} (\tau(\phi_n(e_l)) - \tau(x)), 0 \right).	
	\end{align}
	
	Then define $g^{(l,+)}_\omega,g^{(l,-)}_\omega:\prod_{n=1}^\infty X_n \rightarrow [0,\infty]$ by $(x_n)_{n\in\mathbb{N}} \mapsto \lim_{n\rightarrow\omega} g^{(l,+)}_n(x_n)$ and $(x_n)_{n\in\mathbb{N}} \mapsto \lim_{n\rightarrow\omega} g^{(l,-)}_n(x_n)$ respectively. 
	
	The key observation is that a sequence $x = (x_n)_{n \in \mathbb{N}}$ representing a positive contraction $b \in B_\omega$ satisfies $g^{(l,+)}_\omega(x) = g^{(l,-)}_\omega(x)=0$ if and only if 
	\begin{align}
	\tau(\phi(e_l)) &\leq \tau(b) \leq \tau(\phi(e_l)) + \gamma_l, &\tau \in \overline{T_\omega(B)}^{w*}.
	\end{align}
	
	Let $m \in \mathbb{N}$ and $\epsilon > 0$. By (\ref{eqn:CorrectInTheLimit}), there is $k > m$ such that
	\begin{align}
	\|\phi(e_k)\phi(a_ra_s) - \phi(a_r)\phi(a_s)\| &< \epsilon, &1 \leq r,s \leq m.
	\end{align}
	Let $x = (x_n)_{n \in \mathbb{N}}$ be a sequence of positive contractions in $B$ representing $\phi(e_k)$.  Then $f^{(r,s)}_\omega(x) < \epsilon$ whenever $1 \leq r,s \leq m$. Furthermore, as $k > m$, we have by \eqref{lem:unitisation.gamma.eq} that for any $l \leq m$
	\begin{align}
	\tau(\phi(e_l)) \leq \tau(\phi(e_k)) \leq \theta(\tau) \leq \tau(\phi(e_l)) + \gamma_l, \ \tau \in \overline{T_\omega(B)}^{w*}.
	\end{align}
	Therefore, $g^{(l,+)}_\omega(x) = g^{(l,-)}_\omega(x) = 0$ for $l \leq m$. 
	
	By Kirchberg's Epsilon Test, there exists a sequence of positive contractions $y = (y_n)_{n\in\mathbb{N}}$ in $B$ such that $f^{(r,s)}_\omega(y) = g^{(l,+)}_\omega(y) =g^{(l,-)}_\omega(y) = 0$ for all $r,s,l \in \mathbb{N}$. Let $h$ be the positive contraction in $B_\omega$ represented by $(y_n)_{n\in\mathbb{N}}$. Then $h$ satisfies (\ref{eqn:OrderZeroIndentity1}) for all $a,b \in A_+$ as in (a) and
	\begin{align}
	\tau(\phi(e_l)) &\leq \tau(h) \leq \tau(\phi(e_l)) + \gamma_l, &\tau \in \overline{T_\omega(B)}^{w*}, l \in \mathbb{N}.
	\end{align}
	Letting $l \rightarrow \infty$, we obtain (\ref{eqn:hTrace}) because $\lim_{l \rightarrow \infty} \gamma_l = 0$. The result now follows by Lemma \ref{lem:unitising-orderzero}.
\end{proof}

\section{The Uniqueness Theorem}

In this section, we establish the uniqueness theorem for maps from a $\mathrm{C}^*$-algebra into a $\mathrm{C}^*$-ultrapower, which will be used to bound the nuclear dimension of $\Z$-stable $\mathrm{C}^*$-algebras. This theorem is a non-unital version of \cite[Lemma 4.8]{CETWW} which in turn builds on \cite[Theorem 5.5]{BBSTWW}. For notational convenience, we work with ultrapowers throughout rather than general ultraproducts.

\begin{theorem}[cf.\ {\cite[Theorem 5.5]{BBSTWW}}]
	\label{thm:Uniqueness}
	Let $B$ be a simple, separable, $\mathcal Z$-stable $\mathrm{C}^*$-algebra with CPoU, stable rank one in $B^\sim$, $Q\widetilde{T}(B) = \widetilde{T}_b(B) \neq 0$, and $T(B)$ compact. 
	Let $A$ be a unital, separable, nuclear $\mathrm{C}^*$-algebra, let $\phi_1:A\rightarrow B_\omega$ be a c.p.c.\ order zero map such that $\phi_1(a)$ is full for all non-zero $a \in A$ and the induced map $\bar{\phi}_1: A \to B^\omega$ is a $^*$-homomorphism, and let $\phi_2:A\rightarrow B_\omega$ be a c.p.c.\ order zero map such that
	\begin{equation} \label{eq:uniqueness.thm.st}
	\tau\circ\phi_1=\tau\circ\phi_2^m,\quad \tau\in T(B_\omega),\ m\in\mathbb N,
	\end{equation}
	where order zero functional calculus is used to interpret $\phi_2^m$.\footnote{Suppose $\phi_2(x) = \pi_2(t \otimes x)$ where $\pi_2:C_0(0,1] \otimes A \rightarrow B_\omega$ is a $^*$-homomorphism and $t$ is the canonical generator of $C_0(0,1]$. Then $\phi_2^m(x) = \pi_2(t^m \otimes x)$; see \cite[Corollary 4.2]{WZ09}.} Let $k\in \mathcal Z_+$ be a positive contraction with spectrum $[0,1]$, and define c.p.c.\ order zero maps $\psi_i:A\rightarrow (B\otimes\mathcal Z)_\omega$ by $\psi_i(a):=\phi_i(a)\otimes k$.  Then $\psi_1$ and $\psi_2$ are unitarily equivalent in $(B\otimes\mathcal Z)_\omega^\sim$.
\end{theorem}

The proof of Theorem \ref{thm:Uniqueness} follows by a careful adaptation of the arguments from \cite{BBSTWW, CETWW} to handle the potential non-unitality of $B$. In the subsections that follow, we shall first review the key ingredients of the proof of \cite[Lemma 4.8]{CETWW} and \cite[Theorem 5.5]{BBSTWW} and explain clearly the modifications needed in the non-unital setting. We shall then return to the proof of Theorem \ref{thm:Uniqueness}.

\subsection{The $2 \times 2$ Matrix Trick}

We begin by reviewing the $2 \times 2$ matrix trick, which converts the problem of unitary equivalence of maps into the problem of unitary equivalence of positive elements. The version stated below is very similar to \cite[Lemma 2.3]{BBSTWW}; however, for our applications, we must weaken the stable rank one assumption and we have no need for the Kirchberg algebra case.

\begin{proposition}[{cf.\ \cite[Lemma 2.3]{BBSTWW}}]\label{prop:2x2matrixtrick}
	Let $A$ be a separable, unital $\mathrm C^*$-algebra and $B$ be a separable $\mathrm{C}^*$-algebra. Let $\phi_1,\phi_2:A\rightarrow B_\omega$ be c.p.c.\ order zero maps and $\hat{\phi}_1,\hat{\phi}_2:A\rightarrow B_\omega$ be supporting order zero maps (as in \eqref{eq:Supporting}). Suppose that $B_\omega$ has stable rank one in $B_\omega^\sim$. Let $\pi:A\rightarrow M_2(B_\omega)$ be given by
	\begin{equation}
	\pi(a) := \begin{pmatrix} \hat\phi_1(a) & 0 \\ 0 & \hat\phi_2(a) \end{pmatrix}, \quad a\in A,
	\end{equation}
	and set $C:= M_2(B_\omega)\cap \pi(A)'\cap \{1_{M_2(B_\omega^\sim)}-\pi(1_A)\}^\perp$. If 
	\begin{equation}
	\begin{pmatrix}\phi_1(1_A)&0\\0&0\end{pmatrix}\text{ and }\begin{pmatrix}0&0\\0&\phi_2(1_A)\end{pmatrix}
	\end{equation}
	are unitarily equivalent in $C^\sim$, then $\phi_1$ and $\phi_2$ are unitarily equivalent in $B_\omega^\sim$.
\end{proposition}
\begin{proof}
	Let
	\begin{equation}
	u=\begin{pmatrix} u_{11}&u_{12} \\ u_{21}&u_{22} \end{pmatrix} \in C^\sim
	\end{equation}
	be a unitary implementing the unitary equivalence of the positive elements. Since $B_\omega$ has stable rank one in $B_\omega^\sim$, we have that $u_{21}^*\phi_2(1_A)$ is the limit of invertibles in $B_\omega^\sim$. Hence, by \cite[Lemma 1.20]{BBSTWW} 
	and Kirchberg's Epsilon Test, there is a unitary $w \in B_\omega^\sim$ with $u_{21}^*\phi_2(1_A)= w\,\left|u_{21}^*\phi_2(1_A)\right|$. Arguing exactly as in the proof of \cite[Lemma 2.3]{BBSTWW}, we obtain that $\phi_1(a) = w\phi_2(a)w^*$ for all $a \in A$.   
\end{proof}

\subsection{Property (SI)}

Our goal in this section is to show that c.p.c.\ order zero maps from separable, unital $\mathrm{C}^*$-algebras into ultrapowers of $\mathrm{C}^*$-algebras with compact trace space satisfy property (SI).

The following definition is a variant of \cite[Definition 4.2]{BBSTWW}, which in turn goes back to \cite{MS-CP12}, that allows us to handle cases when the codomain is not unital.

\begin{definition}
	Let $B$ be a simple, separable, $\mathrm{C}^*$-algebra with $Q\widetilde{T}(B) = \widetilde{T}_b(B) \neq 0$. Write $J_{B_\omega}$ for the trace kernel ideal.
	Let $A$ be a separable, unital $\mathrm C^*$-algebra, let $\pi:A\rightarrow B_\omega$ be a c.p.c.\ order zero map, and define 
	\begin{equation} \label{eq.defC}
	C:=B_\omega \cap \pi(A)' \cap \lbrace 1_{B_\omega^\sim}-\pi(1_A) \rbrace^\perp.
	\end{equation}
	The map $\pi$ has \emph{property (SI)} if the following holds. For all $e,f\in C_+$ such that $e\in J_{B_\omega}$, $\|f\|=1$ and $f$ has the property that, for every non-zero $a \in A_+$, there exists $\gamma_a > 0$ such that
	\begin{equation}
	\tau(\pi(a)f^n) > \gamma_a,\quad \tau\in T_\omega(B),\ n\in\mathbb N, \end{equation}
	there exists $s \in C$ such that
	\begin{equation}
	s^*s = e \quad\text{and}\quad fs=s.
	\end{equation}
\end{definition}

The main result of this subsection is that under certain hypotheses, c.p.c.\ maps $A \to B_\omega$ have property (SI). This result is a non-unital version of \cite[Lemma 4.4]{BBSTWW} and
its proof is almost identical to the original proof. 
Since this result is one of the most delicate parts of this work, we include its proof.

\begin{proposition}[cf.\ {\cite[Lemma 4.4]{BBSTWW}}]\label{prop:SI}
	Let $B$ be a simple, separable, $\Z$-stable $\mathrm{C}^*$-algebra with $Q\widetilde{T}(B) = \widetilde{T}_b(B) \neq 0$. 
	Let $A$ be a separable, unital, nuclear $\mathrm C^*$-algebra.
	Then every c.p.c.\ order zero map $\pi:A\rightarrow B_\omega$ has property (SI).
\end{proposition}

\begin{proof}
	Let $\pi:A\rightarrow B_\omega$ be a c.p.c.\ order zero map with $A$ and $B$ as in the statement. Let $C$ be as in \eqref{eq.defC} and set $\overline{C} := C / (C \cap J_{B_\omega}). $
	Let $e,f \in C_+$ and $\gamma_a$ be as in the definition of property (SI). 
	As in the proof of \cite[Lemma 4.4]{BBSTWW}, it is enough to exhibit an element $s \in B_\omega$ approximately satisfying
	\begin{equation}\label{eq:propSI1}
	s^*\pi(a)s = \pi(a)e, \quad \text{for all }a \in A\text{ and}\quad fs=s.
	\end{equation}

	Let $\mathcal{F}\subseteq A$ be a finite subset of contractions and $\epsilon > 0$.
	Since $B$ is $\mathcal Z$-stable, using Lemma \ref{lem:Zfacts}.(ii) we can find a c.p.c.\ order zero map $\alpha:\Z\to B_\omega\cap \pi(A)'\cap\{e,f\}'$ such that $\alpha(1_\Z)$ acts as unit on $\pi(A)$. Therefore, we may define a new  c.p.c.\ map $\tilde{\pi}:A\otimes\Z\to B_\omega$ by setting  $\tilde{\pi}(a\otimes z):=\pi(a)\alpha(z)$. It follows by construction that $\pi(a)=\tilde\pi(a \otimes 1_{\mathcal Z})$ for $a\in A$. By \cite[Corollary 4.3]{WZ09}, $\tilde{\pi}$ is a c.p.c.\ order zero map and note that $e$ and $f$ are elements of the relative commutant $B_\omega\cap \tilde{\pi}(A\otimes\Z)'\cap\{1_{B_\omega^\sim}-\tilde{\pi}(1_{A\otimes\Z})\}^\perp$.
	
	Arguing as in the proof of \cite[Lemma 4.4]{BBSTWW}, for any $b \in (A \otimes \mathcal{Z})_+$, there exists a positive constant $\tilde \gamma_b$ such that
	\begin{equation}\label{eq.relSI1}
	\tau (\tilde \pi(b)f^n) > \tilde{\gamma}_b, \qquad \tau \in T_\omega (B_\omega), \ n \in \mathbb{N}.
	\end{equation}
	
	Next, we will apply \cite[Lemmas 4.7 and 4.8]{BBSTWW} to the unital, separable, nuclear $\mathrm{C}^*$-algebra $A$.
	Set $\mathcal G := \{x \otimes 1_\Z : x \in \mathcal{F} \} \subseteq A\otimes\Z$. Since no irreducible representation of $A \otimes \Z$ contains any compact operator, by \cite[Lemma 4.8]{BBSTWW}
	there exist $L,N \in \mathbb{N}$, pairwise inequivalent pure states $\lambda_1, \ldots, \lambda_L$ on $A \otimes \Z$ and elements $c_i, d_{i,l} \in A \otimes \Z$ for $i=1, \ldots, N, \ l = 1,\ldots, L$ such that
	\begin{equation} \label{eq.relSI2}
	x \approx_\epsilon \sum_{l=1}^L \sum_{i,j=1}^N \lambda_l (d_{i,l}^* x d_{j,l}) c_i^* c_j, \qquad x\in\mathcal G.
	\end{equation}
	
	By \cite[Lemma 4.7]{BBSTWW}, 
	applied to the set $\{d_{i,l}^*xd_{j,l'}:x\in\mathcal G,\ i,j=1,\dots,N,\ l,l'=1,\dots,L\}$, there exist positive contractions $a_1,\dots,a_L\in (A\otimes \mathcal Z)_+$ such that for $l=1,\dots,L$, $\lambda_l(a_l)= 1$ and
	\begin{equation}
	\label{eq:relSIExcision} a_ld_{i,l}^*xd_{j,l}a_l \approx_{\delta} \lambda_l(d_{i,l}^*xd_{j,l})a_l^2,\quad x\in\mathcal G,\ i,j=1,\dots,N,\end{equation}
	and for $l\neq l'$,
	\begin{equation}
	\label{eq:relSIOrthog}
	a_ld_{i,l}^*xd_{j,l'}a_{l'}\approx_{\delta} 0,\quad x\in\mathcal G,\ i,j=1,\dots,N,
	\end{equation}
	with $\delta := \epsilon/(N^2L\max_k\|c_k\|^2)$. Note, the condition $\lambda_l(a_l)= 1$ ensures that the $a_l$ have norm 1.
	
	By hypothesis, $B$ is simple, separable, $\Z$-stable and $Q\widetilde{T}(B) = \widetilde{T}_b(B) \neq 0$. Hence, by Proposition \ref{prop:ZstableToSC}, $B$ has strict comparison of positive elements by bounded traces. Thus, for $l=1,\ldots,L$, we may apply Lemma \ref{lem:relSItrick} with $a_l$ in place of $a$. Let  $S_l \subseteq (A\otimes \mathcal Z)_+ \setminus \{0\}$ denote the countable subset such that the conclusion of Lemma \ref{lem:relSItrick} is satisfied with $a_l$ in place of $a$.
	
	Let $\hat{\pi}:A\otimes \mathcal Z\rightarrow B_\omega\cap \{f\}'$ be a supporting c.p.c.\ order zero map for $\tilde\pi$. As in \cite[Lemma 4.4]{BBSTWW}, using (\ref{eq.relSI1}) and Lemma \ref{lem:LargeTraceSubordinate} twice (taking $x:=0$ and with $S_0:=\tilde\pi(S_1 \cup \cdots \cup S_L)$), we find $t,h \in B_\omega \cap \hat\pi(A\otimes \mathcal Z)'\cap \tilde\pi(A\otimes \mathcal Z)'$ satisfying $h \vartriangleleft t \vartriangleleft f$ and, for every $b\in S_1\cup \cdots \cup S_L$, 
	\begin{equation}
	\tau(\tilde\pi(b)h^n) \geq \tilde{\gamma}_b, \quad \tau \in T_\omega(B),\ n\in\mathbb N. 
	\end{equation}
	By Lemma \ref{lem:relSItrick} (with $\tilde\pi$ in place of $\pi$), there is a contraction $r_l \in B_\omega$ such that  $\tilde\pi(a_l)r_l=tr_l=r_l$ and $r_l^*r_l=e$.
	Using $t \vartriangleleft f\vartriangleleft \tilde\pi(1_A)$, we obtain $\tilde\pi(1_A)r_l=r_l$ for each $l$, and hence,
	\begin{eqnarray}
	\label{eq:relSIrDef}
	r_l^*\tilde\pi(a_l^2)r_l &=& 
	\tilde\pi(1_A)^{1/2}e\tilde\pi(1_A)^{1/2}.
	\end{eqnarray}
	Set
	\begin{equation}
	\label{eq:relSIsDef}
	s := \sum_{l=1}^L\sum_{i=1}^N \hat{\pi}(d_{i,l}a_l)r_l \hat{\pi}(c_i) \in B_\omega.
	\end{equation}
	Using $r_l= tr_l$, $t \vartriangleleft f$ and that $t$ commutes with the image of $\hat\pi$, we can obtain $fs=s$.
	For $x\in \mathcal F$, the calculations of \cite[Lemma 4.4, Equation 4.46]{BBSTWW} shows
	\begin{eqnarray}
	s^*\pi(x)s = \pi(x)e,
	\end{eqnarray}
	as required. 
	Then Kirchberg's Epsilon Test produces an element $s \in B_\omega$ that exactly satisfies \eqref{eq:propSI1}. As in the proof of \cite[Lemma 4.10]{BBSTWW}, $s\in C$.
\end{proof}

\subsection{Structural Results for Relative Commutants}

Combining property (SI) with complemented partitions of unity (CPoU), one can now prove important structural properties for the relative commutant algebras $C:=B_\omega\cap \pi(A)'\cap \lbrace 1_{\widetilde{B}_\omega}-\pi(1_A)\rbrace^\perp$ arising from the $2\times 2$ matrix trick.  

\begin{theorem}[cf.\ {\cite[Lemma 4.7]{CETWW}}]\label{theo:Omnibus}
	Let $B$ be a simple, separable, $\Z$-stable $\mathrm{C}^*$-algebra with $Q\widetilde{T}(B) = \widetilde{T}_b(B) \neq 0$ and $T(B)$ compact. Suppose additionally that $B$ has CPoU.
	Let $A$ be a separable unital nuclear $\mathrm{C}^*$-algebra and $\pi:A\rightarrow B_\omega$ a c.p.c.\ order zero map which induces a $^*$-homomorphism $\overline{\pi}:A\rightarrow B^\omega$. Let 
	\begin{equation}
	C:=B_\omega\cap \pi(A)'\cap \lbrace 1_{\widetilde{B}_\omega}-\pi(1_A)\rbrace^\perp,\quad \overline{C}:=C/(C\cap J_{B_\omega}).
	\end{equation}
	Then:
	\begin{enumerate}
		\item All traces on $C$ factor through $\bar{C}$.
		\item $C$ has strict comparison of positive elements by bounded traces.
		\item The traces on $C$ are the closed convex hull of traces of the form $\tau(\pi(a)\cdot)$ for $\tau\in T(B_\omega)$ and $a\in A_+$ with $\tau(\pi(a))=1$.
	\end{enumerate}
\end{theorem}

First, we discuss two preliminary lemmas, which originate from \cite[Lemma 3.20, Lemma 3.22]{BBSTWW}, and were generalised in \cite[Lemma 4.3, Lemma 4.6]{CETWW} where the newly discovered CPoU was used in place of the earlier methods that required further assumptions on $T(B)$. Both results are proven by checking that these lemmas approximately hold for $\pi_\tau(B^\omega)''$ for any trace in $\tau \in  \overline{T_\omega(B)}^{w*}$, which in turn follows from the fact that $\pi_\tau(B^\omega)''$ is a finite von Neumann algebra,  and then using CPoU to patch local solutions together. In \cite{CETWW}, these results are stated for $B$ unital, but the proofs do not make use of the unit. They only require that $T(B)$ is compact, as this guarantees that $B^\omega$ is unital \cite[Proposition 1.11]{CETWW}.

\begin{lemma}[cf.\ {\cite[Lemma 4.3]{CETWW}}] \label{lem:StrictClosureStrictComp}
	Let $B$ be a separable $\mathrm{C}^*$-algebra with $Q\widetilde{T}(B) = \widetilde{T}_b(B) \neq 0$ and $T(B)$ compact. Suppose $B$ has CPoU. 
	Let $S\subseteq B^\omega$ be a $\|\cdot\|_{2,T_\omega(B)}$-separable and self-adjoint subset, and let $p$ be a projection in the centre of $B^\omega\cap S'$.
	Then $p(B^\omega\cap S')$ has strict comparison of positive elements by bounded traces.
\end{lemma}

\begin{proposition}[cf.\ {\cite[Lemma 4.6]{CETWW}}]\label{prop:CommTraces}
	Let $B$ be a separable $\mathrm{C}^*$-algebra with $Q\widetilde{T}(B) = \widetilde{T}_b(B) \neq 0$ and $T(B)$ compact. Suppose $B$ has CPoU.
	Let $A$ be a separable, unital, nuclear $\mathrm{C}^*$-algebra and $\phi:A\rightarrow B^\omega$ a $^*$-homomorphism.  Set $C:= B^\omega\cap\phi(A)'\cap \lbrace 1_{B^\omega}-\phi(1_A) \rbrace^\perp$. Define $T_0$ to be the set of all traces on $C$ of the form $\tau(\phi(a)\cdot)$ where $\tau \in T(B^\omega)$ and $a \in A_+$ satisfies $\tau(\phi(a))=1$.
	
	Suppose $z \in C$ is a contraction and $\delta>0$ satisfies $|\rho(z)|\leq \delta$ for all $\rho\in T_0$. Write $K := 12 \cdot 12 \cdot (1+\delta)$. Then there exist contractions $w, x_1,\dots,x_{10},y_1,\dots,y_{10} \in C$, such that
	\begin{equation}\label{T3.19:3.37}
	z = \delta w + K\sum_{i=1}^{10} [x_i,y_i].
	\end{equation}
	In particular, $T(C)$ is the closed convex hull of $T_0$.
\end{proposition}

With these preparatory lemmas now established, we explain how to adapt the original proof of \cite[Theorem 4.1]{BBSTWW} to prove Proposition \ref{theo:Omnibus}.

\begin{proof}[Proof of Theorem \ref{theo:Omnibus}]
	For (i), the proof of \cite[Theorem 4.1(i)]{BBSTWW} still works in our situation with the following minor modifications. We use Lemma \ref{lem:LargeTraceSubordinate} instead of \cite[Lemma 1.18]{BBSTWW}, Proposition \ref{prop:SI} in place of \cite[Lemma 4.4]{BBSTWW} and Lemma \ref{lem:RelSurjectivity} in place of \cite[Lemma 1.19]{BBSTWW}.
	
	Similarly, for (ii) we use the proof from \cite[Lemma 3.20]{BBSTWW} with the following modifications.
	Since $B$ is $\mathcal{Z}$-stable, any matrix algebra embeds into $B^\omega \cap \overline{\pi}(A)' \cap \{\overline{c}\}'$ \cite[Proposition 2.3]{CETWW}.
	We use Lemma \ref{lem:StrictClosureStrictComp} to see that $\overline{C}$ has strict comparison of positive elements by traces in place of \cite[Lemma 3.20]{BBSTWW}, and \cite[Lemma 1.8]{CETWW} in place of \cite[Lemma 3.10]{BBSTWW}.
	
	In the same vein, (iii) follows from (i), \cite[Lemma 1.5]{CETWW}, and Proposition \ref{prop:CommTraces}.
\end{proof}

\subsection{Unitary Equivalence of Totally Full Positive Elements}

The main theorem of this section is a non-unital version of the classification of totally full positive elements up to unitary equivalence in relative commutant sequence algebras obtained in \cite[Lemma 5.1]{BBSTWW}.\footnote{Recall that a non-zero $h \in C_+$ is \emph{totally full} if $f(h)$ is full in $C$ for every non-zero $f \in C_0((0,\|h\|])_+$ \cite[Definition 1.1]{BBSTWW}.}

Let us begin by stating the following lemma which can be proved in exactly the same way as \cite[Lemma 5.3]{BBSTWW} since the Robert--Santiago argument (\cite{RS10}) at the core of the proof has no unitality hypothesis. All that is required is to formally replace all occurrences of $1_{B_\omega}$ with $1_{B^\sim_\omega}$, and replace \cite[Lemma 1.17]{BBSTWW} with Lemma \ref{NewEpsLemmaNew}, \cite[Lemma 2.2]{BBSTWW} with Lemma \ref{lem:GapInvertibles}, \cite[Lemma 5.4]{BBSTWW} with Lemma \ref{lem:RSLem}.

\begin{lemma}[{cf.\ \cite[Lemma 5.3]{BBSTWW}}]
	\label{lem:TotallyFullClass}
	Let $B$ be a separable, $\mathcal Z$-stable $\mathrm C^*$-algebra and 
	let $A$ be a separable, unital $\mathrm C^*$-algebra. Let $\pi:A\rightarrow B_\omega$ be a c.p.c.\ order zero map such that
	\begin{equation}
	\label{eq:TotallyFullClassCdef}
	C:=B_\omega \cap \pi(A)' \cap \{1_{B_\omega^\sim}-\pi(1_A)\}^\perp
	\end{equation}
	is full in $B_\omega$.
	
	Assume that every full hereditary subalgebra $D$ of $C$ satisfies the following: if $x \in D$ is such that there exist totally full elements $e_l,e_r \in D_+$ such that $e_l x = xe_r = 0$, then there exists a full element $s \in D$ such that $sx = xs = 0$.
	
	Let $a,b \in C_+$ be totally full positive contractions.
	Then $a$ and $b$ are unitarily equivalent in $C^\sim$ if and only if for every $f \in C_0(0,1]_+$, $f(a)$ is Cuntz equivalent to $f(b)$ in $C$.
\end{lemma}

With this lemma in hand, we can now prove the main theorem of this section.

\begin{theorem}[{cf.\ \cite[Theorem 5.1]{BBSTWW}}]
	\label{thm:TotallyFullClassFinite}
	Let $B$ be a separable, $\mathcal Z$-stable $\mathrm C^*$-algebra with $Q\widetilde{T}(B)=\widetilde{T}_b(B) \neq 0$.
	Let $A$ be a separable, unital $\mathrm C^*$-algebra and let $\pi:A\rightarrow B_\omega$ be a c.p.c.\ order zero map such that
	\begin{equation}\label{e6.1}
	C:=B_\omega \cap \pi(A)' \cap \{1_{B_\omega^\sim}-\pi(1_A)\}^\perp
	\end{equation}
	is full in $B_\omega$ and has strict comparison of positive elements with respect to bounded traces.

	Let $a,b \in C_+$ be totally full positive elements.
	Then $a$ and $b$ are unitarily equivalent in $C^\sim$ if and only if $\tau(a^k) = \tau(b^k)$ for every $\tau \in T(C)$ and $k\in \N$. 
\end{theorem}

\begin{proof}
	 Let $a,b \in C_+$ be totally full positive elements satisfying $\tau(a^k)= \tau(b^k)$ for every $\tau \in T(C)$ and $k \in \mathbb{N}$. Without loss of generality, assume that $a$ and $b$ are contractions. 

After replacing \cite[Lemma 1.22(iv)]{BBSTWW} with Lemma \ref{lem:Zfacts}(iv), part (i) of the proof of \cite[Theorem 5.1]{BBSTWW} shows that the technical hypothesis of Lemma \ref{lem:TotallyFullClass} is satisfied for every full hereditary subalgebra $D \subseteq C$. The argument of part (ii) of the proof of \cite[Theorem 5.1]{BBSTWW} then shows that $f(a)$ is Cuntz equivalent to $f(b)$ for all $f\in C_0(0,1]_+$. (This part of the proof of \cite[Theorem 5.1]{BBSTWW} does not make any use of the unit; only strict comparison is needed.) 
By Lemma \ref{lem:TotallyFullClass}, $a$ and $b$ are unitarily equivalent by unitaries in $C^\sim$. The converse is straightforward.
\end{proof}

\subsection{Proof of the Uniqueness Theorem}

We now have all the ingredients we need for the proof of Theorem \ref{thm:Uniqueness}.

\begin{proof}[Proof of Theorem \ref{thm:Uniqueness}]
	By hypothesis, $\overline{\phi}_1(1_A) \in B^\omega$ is a projection. Hence $d_\tau (\phi_1 (1_A)) = \tau(\phi_1(1_A))$ and we inmediately can conclude that the map $\tau \mapsto d_\tau (\phi_1 (1_A))$ is continuous. 
	Similarly, by equation \eqref{eq:uniqueness.thm.st}, the map $\tau \mapsto d_\tau (\phi_2 (1_A))$ is continuous. Hence, by Lemma \ref{lem:SupportingMapNew}, there exist supporting order zero maps $\hat{\phi}_1, \hat{\phi}_2: A \to B_\omega$ such that 
	\begin{equation}
	\tau( \hat{\phi}_i (a)) = \lim_{m \to \infty} \tau ( \phi_i^{1/m}(a)),\, \qquad a\in A, \, \tau \in T_\omega(B), \ i=1,2,
	\end{equation}
	and the maps $\overline{\hat{\phi}}_i: A \to B^\omega$ are $^*$-homomorphisms.
	In particular,
	\begin{align}
	\tau( \hat{\phi}_2 (a)) \overset{\eqref{eq:uniqueness.thm.st}}{=} \tau(\phi_1(a)), \qquad a \in A, \ \tau \in T_\omega(B).
	\end{align}
	
	By Proposition \ref{prop:SR1inUnitisationUltrapowers}, $B_\omega$ has stable rank one in $B_\omega^\sim$. Thus, we may use the $2 \times 2$ matrix trick (Proposition \ref{prop:2x2matrixtrick}). Recall $\psi_i (a) := \phi_i (a)\otimes k$ and define $\hat{\psi}_1, \hat{\psi}_2  : A \to (B \otimes \mathcal{Z})_\omega$ by $\hat{\psi}_i(a) := \hat{\phi}_i (a) \otimes 1_{\mathcal{Z}}$, with $i=1,2$. It is immediate that $\hat{\psi}_i$ is a supporting order zero map for $\psi_i$. Then define $\pi: A \to M_2 (B_\omega) \subseteq M_2 ((B \otimes \Z)_\omega)$ by
	\begin{equation}
	\pi(a) := \left(
	\begin{matrix}
	\hat{\psi}_1(a) & 0 \\
	0 & \hat{\psi}_2(a)
	\end{matrix}
	\right), \qquad a \in A,
	\end{equation}
	and set $C:= M_2 ((B \otimes \Z)_\omega) \cap \pi(A)' \cap \lbrace 1_{M_2((B \otimes \Z)_\omega^\sim)} - \pi(1_A)\rbrace^\perp$. 
	We will show that 
	\[
	h_1:=\left( \begin{matrix}
	\psi_1(1_A) & 0 \\
	0 & 0
	\end{matrix}\right) \qquad \text{and} \qquad 
	h_2:=\left( \begin{matrix}
	0 & 0 \\
	0 & \psi_2(1_A)
	\end{matrix}\right)
	\]
	are unitarily equivalent in $C^\sim$.
	For non-zero $a \in A$, observe that
	\begin{equation}
	0 \leq \left( \begin{matrix}
	\psi_1(a) & 0 \\
	0 & 0
	\end{matrix}\right) \leq \left( \begin{matrix}
	\hat{\psi}_1(a) & 0 \\
	0 & \hat{\psi}_2(a)
	\end{matrix}\right) = \pi(a),
	\end{equation}
	and using that $\psi_1 (a)$ is full in $(B \otimes \mathcal{Z})_\omega$ since $\phi_1(a)$ is full, we conclude that $\pi(a)$ is full in $M_2((B \otimes \Z)_\omega)$. By construction, the induced map $\overline{\pi}:A \to M_2(B^\omega)$ is a $^*$-homomorphism. Thus, by Theorem \ref{theo:Omnibus}, $C$ has strict comparison. 
	
	Notice that $h_1 \in C$ is full in $M_2(B_\omega)$, and hence $C$ is also full in $M_2(B_\omega)$. Let $\rho$ be a trace on $C$ of the form $\tau (\pi(x) \cdot)$ where $\tau \in T(M_2((B \otimes \Z)_\omega)), x \in A_+$ and $\tau (\pi(x))=1$. Set a trace $\tilde{\tau}$ on $B_\omega$ by $\tilde{\tau}(b):= \tau(1_{M_2} \otimes b \otimes 1_{\mathcal{Z}})$. Thus, as in \cite[Theorem 5.5, equation (5.41)]{BBSTWW},
	\begin{align}
	\rho(h_1^m) = \frac{1}{2} \tau_\mathcal{Z}(k^m) = \rho(h_2^m), \qquad m \in \mathbb{N}.
	\label{eq.uniq.thm.1}
	\end{align}
	By Theorem \ref{theo:Omnibus}, equation \eqref{eq.uniq.thm.1} holds for any trace on $C$.
	
	An standard strict comparison argument shows that $f(h_1)$ and $f(h_2)$ are full in $C$ for any $f \in C_0(0,1]_+$, so $h_1$ and $h_2$ are totally full. By Theorem \ref{thm:TotallyFullClassFinite}, $h_1$ is unitarily equivalent to $h_2$ in $C^\sim$. By the $2 \times 2$ matrix trick (Proposition \ref{prop:2x2matrixtrick}), $\psi_1$ and $\psi_2$ are unitarily equivalent in  $(B\otimes\mathcal Z)_\omega^\sim$.
\end{proof}

\section{Nuclear Dimension and $\Z$-Stability} \label{section.nucleardim}

In this section, we prove Theorems \ref{thm:NewMain2} and \ref{thm:NewMain}, and deduce Corollaries \ref{cor:NewTrichotomy} and \ref{cor:Newclassification}.

\begin{theorem}\label{thm:Main2}
	Let $A$ be a simple, separable, nuclear and $\mathcal{Z}$-stable $\mathrm{C}^*$-algebra. Then $\dim_{\mathrm{nuc}} A \leq 1$.
\end{theorem}
\begin{proof}
	
	By Theorem \ref{thm:simple.reduction}, either $A$ is stably isomorphic to a unital $\mathrm{C}^*$-algebra $B$, or 
	$A$ is stably isomorphic to a stably projectionless $\mathrm{C}^*$-algebra $A_0$  with $Q\widetilde{T}(A_0) = \widetilde{T}_b(A_0) \neq 0$ and $T(A_0)$ compact.   
	
	The stably unital case follows immediately from \cite[Theorem B]{CETWW} together with Proposition \ref{thm:StableIsoInvariants}.	Indeed, 
	if $A$ is stably isomorphic to a unital $\mathrm{C}^*$-algebra $B$, then $B$ is also simple, separable, nuclear and $\mathcal{Z}$-stable by Proposition \ref{thm:StableIsoInvariants}.
	Hence, $\dimnuc B \leq 1$ by \cite[Theorem B]{CETWW}. Therefore, $\dimnuc A \leq 1$ by a second application of Proposition \ref{thm:StableIsoInvariants}.
	
	We now consider the case when $A$ is stably isomorphic to a stably projectionless $\mathrm{C}^*$-algebra $A_0$ with $Q\widetilde{T}(A_0) = \widetilde{T}_b(A_0) \neq 0$ and $T(A_0)$ compact.  
	By Proposition \ref{thm:StableIsoInvariants}, $A_0$ is simple, separable, nuclear and $\mathcal{Z}$-stable. Since $A_0$ is stably projectionless and $\mathcal{Z}$-stable, $A_0$ has stable rank one in $A_0^\sim$ by Theorem \ref{thm:ZstableProjectionless}. Furthermore, $A_0$ has CPoU by Theorem \ref{thm:CPoU}.
	
	In light of Proposition \ref{thm:StableIsoInvariants}, it suffices to prove that $\dim_{\mathrm{nuc}} A_0 \leq 1$. We now show this using the same fundamental strategy of \cite{BBSTWW} (taking into account the modification introduced in \cite{CETWW}). We shall estimate the nuclear dimension of the first factor embedding $j:A_0 \rightarrow A_0 \otimes \mathcal{Z}$, $j(x)= x\otimes 1_{\mathcal{Z}}$ in the sense of \cite[Definition 2.2]{TW14}. Since $A_0$ is $\mathcal{Z}$-stable and $\mathcal{Z}$ is strongly self-absorbing, we have $\dim_{\mathrm{nuc}}(A_0)=\dim_{\mathrm{nuc}}(j)$; see \cite[Proposition 2.6]{TW14}.
	
	Let $\iota:A_0 \rightarrow (A_0)_\omega$ be the canonical embedding. Let $h$ be a strictly positive contraction in $A_0$, and let $(e_n)_{n \in \N}$ be the approximate identity given by $e_n := h^{1/n}$. Then $\lim_{n\rightarrow\infty}\tau(e_n) = 1$ for all $\tau \in T(A_0)$. Since $T(A_0)$ is compact, $\tau \circ \iota \in T(A_0)$ for all $\tau \in T_\omega(A_0)$ and so for all $\tau \in \overline{T_\omega(A_0)}^{w*}$. It follows that
	\begin{equation}
	\lim_{n\rightarrow\infty} \tau(\iota(e_n)) = 1, \quad \quad \tau \in \overline{T_\omega(A_0)}^{w*}.
	\end{equation} 
	Therefore, applying Lemma \ref{lem:unitising-orderzero}, we obtain a c.p.c.\ order zero extension $\iota^\sim:A_0^\sim \rightarrow (A_0)_\omega$ with $\tau(\iota^\sim(1_{A_0^\sim})) = 1$ for all $\tau \in \overline{T_\omega(A_0)}^{w*}$. Writing $\overline{\iota^\sim}:A_0^\sim \rightarrow A_0^\omega$ for the induced map into the uniform tracial ultrapower, we observe that $1_{A_0^\omega} - \overline{\iota^\sim}(1_{A_0^\sim})$ is a positive element in $A_0^\omega$ that vanishes on all limit traces, so must be zero. Hence, $\overline{\iota^\sim}$ is a unital c.p.c.\ order zero map, so must be a unital $^*$-homomorphism.

Let $(\phi_n:A_0 \rightarrow A_0)_{n=1}^\infty$ be the sequence of c.p.c.\ maps constructed in Lemma \ref{lem:NiceFactoring}, which factorize as $\eta_n\circ\theta_n$ through finite dimensional algebras $F_n$ as in \eqref{eq:NiceFactoring1}. By construction, the induced map $\Phi:A_0 \rightarrow (A_0)_\omega$ is c.p.c.\ order zero and the induced map $\overline{\Phi}: A_0 \to A_0^\omega$ agrees with the diagonal inclusion $\overline{\iota}:A_0 \rightarrow A_0^\omega$. It follows that $\tau \circ \Phi = \tau \circ \iota$ for all $\tau \in \overline{T_\omega(A)}^{w*}$. Hence,
	\begin{equation}
	\lim_{n\rightarrow\infty} \tau(\Phi(e_n)) = 1, \quad \quad \tau \in \overline{T_\omega(A_0)}^{w*}.
	\end{equation}  
	Therefore, applying Lemma \ref{lem:unitising-orderzero} again, we obtain a c.p.c.\ order zero extension $\Phi^\sim:A_0^\sim \rightarrow (A_0)_\omega$ with $\tau(\Phi^\sim(1_{A_0^\sim})) = 1$ for all $\tau \in \overline{T_\omega(A_0)}^{w*}$. Arguing as before, $\overline{\Phi^\sim}:A_0^\sim \rightarrow A_0^\omega$ is a unital $^*$-homomorphism. In fact, we have $\overline{\Phi^\sim} = \overline{\iota^\sim}$ since both maps agree on $A_0$ by construction and are unital.  
	
	We are almost ready to apply Theorem \ref{thm:Uniqueness} to the c.p.c.\ order zero maps $\iota^\sim$ and $\Phi^\sim$. We observe that $A_0$ is a simple, separable, $\Z$-stable with CPoU, stable rank one in $A_0^\sim$, $Q\widetilde{T}(A_0) = \widetilde{T}_b(A_0) \neq 0$, and $T(A_0)$ compact; that $A_0^\sim$ is unital, separable and nuclear; and that both maps induce a unital $^*$-homomorphism $\overline{\iota^\sim} = \overline{\Phi^\sim}:A_0^\sim \rightarrow A_0^\omega$. Since $\overline{\iota^\sim} = \overline{\Phi^\sim}$ and both maps are $^*$-homomorphisms, we have
	\begin{equation}
	\tau \circ \iota = \tau \circ \Phi^m, \quad \quad \tau \in \overline{T_\omega(A_0)}^{w*}, \ m \in \N. 
	\end{equation}
	The tracial condition \eqref{eq:uniqueness.thm.st} follows because $T_\omega(A_0)$ is dense in $T((A_0)_\omega)$ by Theorem \ref{thm:NoSillyTraces}.
	
	Before we may apply Theorem \ref{thm:Uniqueness}, we must show that $\iota^\sim(x)$ is full for all non-zero $x \in A_0^\sim$. By Proposition \ref{prop:ZstableToSC}, $A_0$ has strict comparison by bounded traces because $A_0$ is simple, separable, $\Z$-stable and $Q\widetilde{T}(A_0) = \widetilde{T}_b(A_0) \neq 0$. Hence, $(A_0)_\omega$ has strict comparison in the sense of Lemma \ref{lem:StrictCompLimTraces}.
	
	Using that $A_0$ is simple and $T(A_0)$ is compact, the minimum $\gamma_a := \min_{\tau \in T(A)}\tau(a)$ exists and is strictly positive for any non-zero $a \in (A_0)_{+,1}$. Since $\tau \circ \iota \in T(A_0)$ for any $\tau \in \overline{T_\omega(A_0)}^{w*}$, we have $d_\tau(\iota(a)) \geq \gamma_a$ for any $\tau \in \overline{T_\omega(A_0)}^{w*}$. Hence, $\iota(a)$ is full in $(A_0)_\omega$ using Lemma \ref{lem:StrictCompLimTraces}. 
	
	For any non-zero $x \in A_0^\sim$, the ideal $I_x$ of $A_0^\sim$ generated by $x$ contains a non-zero positive contraction $a \in A_{+,1}$. A simple computation using supporting order zero maps shows that the ideal of $(A_0)_\omega$ generated by $\iota^\sim(x)$ contains $\iota^\sim(I_x)$, which is full since it contains the full element $\iota^\sim(a)$. Hence, $\iota^\sim(x)$ is full in $(A_0)_\omega$. 
	
	Fix a positive contraction $k \in \Z_+$ of full spectrum. Applying Theorem \ref{thm:Uniqueness} to the maps $\iota^\sim$ and $\Phi^\sim$, we obtain unitaries $w^{(0)},w^{(1)}\in (A_0 \otimes\Z)_\omega^\sim$ such that
	\begin{align}
	x \otimes k &=w^{(0)}(\Phi(x) \otimes k)w^{(0)}{}^*,\\
	x \otimes (1_\Z-k)&=w^{(1)}(\Phi(x)\otimes (1_\Z-k))w^{(1)}{}^*,\quad x\in A.
	\end{align}
	Choose representing sequences $(w^{(0)}_n)_{n=1}^\infty$ and $(w^{(1)}_n)_{n=1}^\infty$ of unitaries in $(A_0 \otimes \Z)^\sim$ for $w^{(0)}$ and $w^{(1)}$, respectively.  
	We have c.p.c.\ maps $\theta_n\oplus\theta_n:A_0 \rightarrow F_n\oplus F_n$, and $\tilde{\eta}_n:F_n\oplus F_n\rightarrow A_0 \otimes \Z$, where 
	\begin{equation}
	\tilde{\eta}_n(y_0,y_1):=w^{(0)}_n(\eta_n(y_0)\otimes k)w_n^{(0)}{}^*+w_n^{(1)}(\eta_n(y_1)\otimes (1_\Z-k))w_n^{(1)}{}^*.
	\end{equation}
	Hence, $j(x)$ is the limit, as $n \to \omega$, of $(\tilde{\eta}_n\circ(\theta_n\oplus\theta_n)(x))_{n=1}^\infty$ and, since $\tilde{\eta}_n$ is the sum of two c.p.c.\ order zero maps, $\dimnuc(j)\leq 1$.
\end{proof}

\begin{theorem}\label{cor:TW2}
	Let $A$ be a non-elementary, simple, separable, nuclear $\mathrm{C}^*$-algebra. Then $A$ has finite nuclear dimension if and only if it is $\mathcal{Z}$-stable.
\end{theorem}
\begin{proof}
	Let $A$ be a non-elementary, simple, separable, nuclear $\mathrm{C}^*$-algebra. If $A$ is $\Z$-stable, then $\dimnuc(A) \leq 1 < \infty$ by Theorem \ref{thm:Main2}. Conversely, if $\dimnuc(A) < \infty$, then $A$ is $\Z$-stable by \cite[Theorem 8.5]{Ti14}
\end{proof}

\begin{corollary}\label{cor:Trichotomy2}
	The nuclear dimension  
	of a simple $\mathrm{C}^*$-algebra is $0, \, 1$ or $\infty$.
\end{corollary}
\begin{proof}
	Let $A$ be a simple, separable $\mathrm{C}^*$-algebra with finite nuclear dimension. Then, in particular, $A$ is nuclear. If $A$ is elementary, then $\dimnuc(A) = 0$; otherwise, $A$ is $\Z$-stable by Corollary \ref{cor:TW2}. Hence, $\dimnuc(A) \leq 1$ by Theorem \ref{thm:Main2}. The non-separable case follows from the separable one as in the proof of \cite[Corollary C]{CETWW}.
\end{proof}

In \cite[Theorem 5.2.2]{Ell96}, a stably projectionless, simple, separable, nuclear C$^*$-algebra with a unique trace, $K_0 = \mathbb{Z}$ and $K_1 = 0$ is constructed as a limit of 1 dimensional non-commutative CW complexes. By \cite[Theorem 1.4]{Go17}, there is a unique C$^*$-algebra with theses properties that has finite nuclear dimension and satisfies the UCT. This C$^*$-algebra is denoted $\mathcal{Z}_0$ \cite[Definition 8.1]{Go17}, reflecting its role as an stably projectionless analogue of the Jiang--Su algebra $\mathcal{Z}$. An important further property of $\mathcal{Z}_0$, which follows from its construction, is that $\mathcal{Z}_0$ is $\mathcal{Z}$-stable \cite[Remark 7.3, Definition 8.1]{Go17}. 

It has recently been shown that simple, separable $\mathrm{C}^*$-algebras which satisfy the UCT and have finite nuclear dimension are classified up to stabilisation with $\mathcal{Z}_0$ by the Elliott invariant \cite[Theorem 1.2]{Go17}. The appropriate form of the Elliott invariant in this setting is detailed in \cite[Definition 2.9]{Go17}. In light of the main result of this paper, we can weaken the hypothesis of finite nuclear dimension in \cite[Theorem 1.2]{Go17} to that of nuclearity.

\begin{corollary}[{cf.\ \cite[Theorem 1.2]{Go17}}]
	Let $A$ and $B$ be simple, separable, nuclear $\mathrm{C}^*$-algebras which satisfy the UCT. Then 
	\begin{equation}
	A \otimes \mathcal{Z}_0 \cong B \otimes \mathcal{Z}_0 \text{ if and only if } \mathrm{Ell}(A \otimes \mathcal{Z}_0) \cong \mathrm{Ell}(B \otimes \mathcal{Z}_0). \notag
	\end{equation}
\end{corollary}

\begin{proof}
	Since $A$ and $\mathcal{Z}_0$ are simple, separable and nuclear, so is $A\otimes \mathcal{Z}_0$. Using that $\Z_0$ is $\Z$-stable, it follows that $A\otimes \mathcal{Z}_0$ is $\mathcal{Z}$-stable. 
	Therefore, $\dim_{\mathrm{nuc}} A \otimes \mathcal{Z}_0 \leq 1$ by Theorem \ref{thm:NewMain2}. Similarly, $\dim_{\mathrm{nuc}} B \otimes \mathcal{Z}_0 \leq 1$.  The result now follows from \cite[Theorem 1.2]{Go17}.
\end{proof}

\section{Decomposition Rank and $\Z$-Stability}

Using the machinery developed to prove Theorem \ref{thm:NewMain2}, we can also prove similar results for the decomposition rank of simple $\Z$-stable $\mathrm{C}^*$-algebras under suitable finiteness and quasidiagonality assumptions. To this end, we recall the definition of quasidiagonality for tracial states.   

\begin{definition}[cf.\ {\cite[Definition 3.3.1]{Br06}}]\label{dfn:QDtraces}
Let $A$ be a $\mathrm{C}^*$-algebra. A tracial state $\tau \in T(A)$ is \emph{quasidiagonal} if there exists a net\footnote{When $A$ is separable, one can work with sequences instead of general nets.}  of c.p.c.\ maps $\phi_n:A \rightarrow M_{k_n}(\mathbb{C})$ with $\|\phi_n(ab) - \phi_n(a)\phi_n(b)\| \rightarrow 0$ and $\mathrm{tr}_{k_n}(\phi_n(a)) \rightarrow \tau(a)$.
\end{definition}

 In the unital case, the c.p.c.\ maps in Definition \ref{dfn:QDtraces} can be taken to be unital (see the proof of \cite[Lemma 7.1.4]{Br08}). Moreover, a trace $\tau \in T(A)$ is quasidiagonal if and only if its extension to $A^\sim$ is quasidiagonal  \cite[Proposition 3.5.10]{Br06}. We write $T_{QD}(A)$ for the set of all quasidiagonal tracial states on $A$.

We can now state a decomposition rank version of Theorem \ref{thm:NewMain2}. 

\begin{theorem}\label{thm:dr}
	Let $A$ be a simple, separable, nuclear and $\mathcal{Z}$-stable $\mathrm{C}^*$-algebra. Suppose further that $A$ is stably finite and that $T(B) = T_{QD}(B)$ for all non-zero hereditary subalgebras $B \subseteq A \otimes \K$. Then $\dr(A) \leq 1$.
\end{theorem}
\begin{proof}
By Theorem \ref{thm:simple.reduction}, either $A$ is stably isomorphic to a unital $\mathrm{C}^*$-algebra $B$, or $A$ is stably isomorphic to a stably projectionless $\mathrm{C}^*$-algebra $A_0$  with $Q\widetilde{T}(A_0) = \widetilde{T}_b(A_0) \neq 0$ and $T(A_0)$ compact.
	
In the first case, $B$ is simple, separable, nuclear and $\Z$-stable by Proposition \ref{thm:StableIsoInvariants}. Moreover, $B$ is finite and $T(B) = T_{QD}(B)$ by our additional hypotheses on $A$. Hence $\dr(B) \leq 1$ by \cite[Theorem B]{CETWW}. By Proposition \ref{thm:StableIsoInvariants} once more, $\dr(A) \leq 1$.

In the second case, we have that $T(A_0) = T_{QD}(A_0)$ by our additional hypotheses on $A$, so in the proof of Theorem \ref{thm:Main2} the maps $\theta_n$ from Lemma \ref{lem:NiceFactoring} can be taken to be approximately multiplicative. Therefore, $\dr(A_0)\leq 1$ by \cite[Lemma 1.9]{BGSW}. Hence, $\dr(A) \leq 1$ by Proposition \ref{thm:StableIsoInvariants}.   
\end{proof}
\begin{remark}
If $A$ is a simple, separable, nuclear $\mathrm{C}^*$-algebra in the UCT class, then $T(B) = T(B)_{QD}$ for all hereditary subalgebras $B \subseteq A \otimes \K$ by \cite[Theorem A]{TWW17} since the UCT class is closed under stable isomorphism. 
\end{remark}  

As with nuclear dimension, we obtain a trichotomy result for decomposition rank as a corollary of Theorem \ref{thm:dr}.

\begin{corollary}
The decomposition rank of a simple $\mathrm{C}^*$-algebra is $0, \, 1$ or $\infty$.
\end{corollary}
\begin{proof}

Elementary $\mathrm{C}^*$-algebras have decomposition rank zero, so are covered by this result.

Let $A$ be a non-elementary, simple, separable $\mathrm{C}^*$-algebra with finite decomposition rank. Then $A$ has finite nuclear dimension, and so is $\Z$-stable by \cite[Corollary 8.6]{Ti14}. 
Since $\dr(A) < \infty$, $A$ is stably finite and $T(A) = T_{QD}(A)$.\footnote{One can reduce to the unital case because $\dr(A) = \dr(A^\sim)$ \cite[Proposition 3.4]{KW04}. Then $T(A) = T_{QD}(A)$ by \cite[Proposition 8.5]{BBSTWW}. Stably finiteness of $A$ follows from \cite[Proposition 5.1]{KW04} and \cite[Theorem 7.1.15]{Br08} for example.} Moreover, by Corollary \ref{cor:Brown} and Proposition \ref{thm:StableIsoInvariants}, $\dr(B) = \dr(A) < \infty$ for any non-zero hereditary subalgebra $B \subseteq A \otimes \K$. Therefore, we have $T(B) = T_{QD}(B)$. Now, $\dr(A) \leq 1$ by Theorem \ref{thm:dr}.

The non-separable case follows from the separable case as in the proof of \cite[Corollary C]{CETWW} since the proof of \cite[Proposition 2.6]{WZ10} works equally well for decomposition rank.
\end{proof}

\appendix

\section{Non-unital Lemmas}

The purpose of this appendix is to state appropriate non-unital versions of the technical lemmas from \cite{BBSTWW}. In cases where substantial modifications to the proof are required, we give full details. In cases where the modifications are trivial, we refer the reader to the proof of the corresponding result from \cite{BBSTWW} and explain the modifications in a remark.

We begin with the existence of supporting order zero maps.   

\begin{lemma}[{cf.\ \cite[Lemma 1.14]{BBSTWW}}]
	\label{lem:SupportingMapNew}
	Let $A,B_n$ be $\mathrm C^*$-algebras with $A$ separable and unital, set $B_\omega := \prod_\omega B_n$, and suppose that $S\subseteq B_\omega$ is separable and self-adjoint.
	Let $\phi:A \to B_\omega\cap S'$ be a c.p.c.\ order zero map.
	Then there exists a c.p.c.\ order zero map $\hat\phi:A \to B_\omega\cap S'$ such that
	\begin{align}
	\label{eq:Supporting}
	\phi(ab)&=\hat{\phi}(a)\phi(b)=\phi(a)\hat{\phi}(b),\quad a,b\in A.
	\end{align}
	
	Suppose now that $T(B_n)$ is non-empty for all $n \in \N$. If the map $\tau \mapsto d_\tau(\phi(1_A))$ from $\overline{T_\omega(B_\omega)}^{w*}$ to $[0,1]\subseteq\mathbb R$ is continuous (with respect to the weak$^*$-topology) then we can, in addition, arrange that
	\begin{align}
	\label{eq:SupportingTrace}
	\tau(\hat{\phi}(a)) &= \lim_{m\to\infty} \tau(\phi^{1/m}(a)), \quad a\in A_+,\ \tau \in T_\omega(B_\omega),
	\end{align}
	where order zero map functional calculus is used to interpret $\phi^{1/m}$. In this case, the induced map $\overline{\hat{\phi}}:A \to B^\omega$ is a $^*$-homomorphism.
\end{lemma}
\begin{proof}[Remarks]
	The proof of \cite[Lemma 1.14]{BBSTWW} only actually requires continuity of $\tau \mapsto d_\tau(\phi(1_A))$ on $\overline{T_\omega(B_\omega)}^{w*}$ (as opposed to $T(B_\omega)$) and $\overline{T_\omega(B_\omega)}^{w*}$ is compact in the non-unital case too. There is no further use of the unitality of the $B_n$ in the proof of \cite[Lemma 1.14]{BBSTWW}.  
\end{proof}

We now record some more straightforward applications of the Kirchberg's Epsilon Test. These results are almost identical to those proven in \cite[Section 1]{BBSTWW}. However, we shall need slightly more general statements because we wish to apply them to the algebras of the form $B_\omega \cap S' \cap \lbrace 1_{B_\omega^\sim}-d \rbrace^\perp$.   

\begin{lemma}[{cf.\ \cite[Lemma 1.16]{BBSTWW}}]
	\label{lem:ActAsUnitNew} 
	Let $(B_n)_{n=1}^\infty$ be a sequence of $\mathrm C^*$-algebras and set $B_\omega := \prod_\omega B_n$.
	Let $S_1,S_2$ be separable self-adjoint subsets of $B_\omega^\sim$, and let $T$ be a separable subset of $B_\omega \cap S_1' \cap S_2^\perp$.
	Then there exists a contraction $e \in (B_\omega \cap S_1' \cap S_2^\perp)_{+}$ that acts as a unit on $T$, i.e., such that $et=te=t$ for every $t\in T$.
\end{lemma}
\begin{proof}[Remarks]
	The only change to the statement is that $S_1,S_2$ are subsets of $B_\omega^\sim$ (as opposed to $B_\omega$). The proof is not affected. 
\end{proof}

\begin{lemma}[{cf.\ \cite[Lemma 1.17]{BBSTWW}}]\label{NewEpsLemmaNew} 
	Let $(B_n)_{n=1}^\infty$ be a sequence of $\mathrm C^*$-algebras and set $B_\omega := \prod_\omega B_n$.
	Let $S_1,S_2$ be separable self-adjoint subsets of $B_\omega^\sim$, and set $C:=B_\omega\cap S_1'\cap S_2^\perp$.
	\begin{enumerate}
		\item Let $h_1,h_2\in C_+$. Then $h_1$ and $h_2$ are unitarily equivalent via a unitary from $C^\sim$ if and only if they are approximately unitarily equivalent, i.e., for any $\epsilon>0$ there exists a unitary $u\in C^\sim$ with $uh_1u^*\approx_\epsilon h_2$.
		\item Let $a\in C$.  Then there exists a unitary $u\in C^\sim$ with $a=u|a|$ if and only if for each $\epsilon>0$ there exists a unitary $u\in C^\sim$ with $a\approx_\epsilon u|a|$.
		\item Let $h_1,h_2 \in C_+$. Then $h_1$ and $h_2$ are Murray-von Neumann equivalent if and only if they are approximately Murray-von Neumann equivalent, i.e., for any $\epsilon>0$ there exists $x\in C$ with $xx^*\approx_\epsilon h_1$ and $x^*x\approx_\epsilon h_2$.
	\end{enumerate}
\end{lemma}
\begin{proof}[Remarks]
	The statement of \cite[Lemma 1.17]{BBSTWW} uses the convention that $C^\sim := C$ when $C$ is already unital. In this paper, we use the convention that a new unit is still adjoined, so $C^\sim \cong C \oplus \mathbb{C}$ when $C$ is unital. The choice of convention does not affect the validity of the lemma.\footnote{For example, if $C$ is unital $h_1,h_2$ are unitary equivalent in $C$ if and only if they are unitary equivalent in $C \oplus \mathbb{C}$.} Apart from this, the only change to the statement is that $S_1,S_2$ are subsets of $B_\omega^\sim$ (as opposed to $B_\omega$), which does not affect the proof. 
\end{proof}

\begin{lemma}[cf.\ {\cite[Lemma 1.18]{BBSTWW}}]
	\label{lem:LargeTraceSubordinate}
	Let $(B_n)_{n=1}^\infty$ be a sequence of $\mathrm C^*$-algebras with $T(B_n)$ non-empty for each $n\in\N$. Write $B_\omega := \prod_\omega B_n$.  Let $S_0$ be a countable self-adjoint subset of $(B_\omega)_+$ and let $T$ be a separable self-adjoint subset of $B_\omega$.
	If $x,f \in (B_\omega \cap S_0'\cap T')_+$ are contractions with $x\vartriangleleft f$ and with the property that for all $a \in S_0$ there exists $\gamma_a \geq 0$ such that $\tau(af^m) \geq \gamma_a$ for all $m\in \N,\ \tau \in T_\omega(B_\omega)$, then there exists a contraction $f' \in (B_\omega \cap S_0'\cap T')_+$ such that $x\vartriangleleft f' \vartriangleleft f$ and $\tau(a(f')^m) \geq \gamma_a$ for all $m\in\N,\ \tau \in T_\omega(B_\omega)$, and $a\in S_0$.
	
	If each $B_n$ is simple, separable, $\mathcal Z$-stable and $Q\widetilde{T}(B_n) = \widetilde{T}_b(B_n) \neq 0$ for all $n \in \N$, then the above statement holds with $T(B_\omega)$ in place of $T_\omega(B_\omega)$.
\end{lemma}
\begin{proof}[Remarks]
	The only change to the proof of \cite[Lemma 1.18]{BBSTWW} is to replace $\min_{\tau \in T(B_n)}$ with $\inf_{\tau \in T(B_n)}$ in \cite[Equation (1.34)]{BBSTWW}, as the minimum need not exist in the non-unital case. The final sentence follows since $T_\omega(B_\omega)$ is weak$^*$-dense in $T_\omega(B_\omega)$ under the additional hypotheses by Theorem \ref{thm:NoSillyTraces}.
\end{proof}

\begin{lemma}[cf.\ {\cite[Lemma 1.19]{BBSTWW}}]
	\label{lem:RelSurjectivity}
	Let $(B_n)_{n=1}^\infty$ be a sequence of separable $\mathrm C^*$-algebras with $T(B_n)\neq\emptyset$ for each $n\in\N$ and set $B_\omega:=\prod_\omega B_n$.
	Let $A$ be a separable, unital $\mathrm C^*$-algebra and let $\pi:A\rightarrow B_\omega$ be a  c.p.c.\ order zero map such that $\pi(1_A)$ is full and the induced map $\bar{\pi}:A \to B^\omega$ is a $^*$-homomorphism.
	Define $C:=B_\omega \cap \pi(A)' \cap \{1_{B_\omega^\sim}-\pi(1_A)\}^\perp$.
	Let $S \subseteq C$ be a countable self-adjoint subset and let $\bar{S}$ denote the image of $S$ in $B^\omega$.
	\begin{enumerate}
		\item[(i)] Then the image of $C\cap S'$ in $B^\omega$ is precisely
		\begin{eqnarray}
		\lefteqn{ \bar{\pi}(1_A)\left(B^\omega \cap \bar{\pi}(A)' \cap
			\bar{S}'\right) } \nonumber \\
		& = &B^\omega \cap
		\bar{\pi}(A)'\cap \bar{S}'\cap \{1_{(B^\omega)^\sim}-\bar{\pi}(1_A)\}^\perp,
		\end{eqnarray}
		a $\mathrm C^*$-subalgebra of $B^\omega$ with unit $\bar{\pi}(1_A)$.
		\item[(ii)] Let $\tau\in T_\omega(B_\omega)$ be a limit trace and $a\in A_+$ and form the tracial functional $\rho:=\tau(\pi(a)\cdot)$ on $C$.
		Then  $\|\rho\|=\tau(\pi(a))$.  If each $B_n$ is additionally simple, $\mathcal Z$-stable and $Q\widetilde{T}(B_n) = \widetilde{T}_b(B_n) \neq 0$ for all $n \in \N$, then this holds for all traces $\tau\in T(B_\omega)$.
	\end{enumerate}
\end{lemma}
\begin{proof}[Remarks]
	The proof of \cite[Lemma 1.19]{BBSTWW} does not need the $B_n$ to be unital. Note that the notation $\{1_{B_\omega^\sim}-\pi(1_A)\}^\perp$ is just an alternative notation for subalgebra on which $\pi(1_A)$ acts as a unit, and similarly for $\{1_{(B^\omega)^\sim}-\bar{\pi}(1_A)\}^\perp$. The final sentence of (ii) follows since $T_\omega(B_\omega)$ is weak$^*$-dense in $T_\omega(B_\omega)$ under the additional hypotheses by Theorem \ref{thm:NoSillyTraces}.
\end{proof}

Next, we consider some properties of ultraproducts of separable, $\Z$-stable $\mathrm{C}^*$-algebras.  
\begin{lemma}[cf.\ {\cite[Lemma 1.22]{BBSTWW}}]
	\label{lem:Zfacts}
	Let $(B_n)_{n\in\N}$ be a sequence of separable, $\Z$-stable $\mathrm C^*$-algebras and set $B_\omega:=\prod_\omega B_n$. Then:
	\begin{itemize}
		
		\item[(i)]\label{lem:Zfacts2} If $S\subseteq B_\omega$ is separable, then there exist isomorphisms $\phi_n:B_n \to B_n \otimes \Z$ such that the induced isomorphism $\Phi:B_\omega \to \prod_\omega (B_n \otimes \Z)$ maps $x\in S$ to $x \otimes 1_{\mathcal Z} \in (\prod_\omega B_n) \otimes \Z\subseteq \prod_\omega (B_n\otimes\Z)$. 
		
		\item[(ii)]\label{lem:Zfacts1} Let $S_1,S_2 \subseteq B_\omega^\sim$ be separable and self-adjoint. For any separable subset $ T \subseteq B_\omega \cap S_1' \cap S_2^\perp$, there exists a c.p.c.\ order zero map $\psi:\Z \rightarrow B_\omega \cap S_1' \cap S_2^\perp \cap T'$ such that $\psi(1_\Z)$ acts as a unit on $T$.
		\item[(ii')]\label{lem:Zfacts1'} Let $S_1,S_2 \subseteq B_\omega^\sim$ be separable and self-adjoint. For any separable subalgebra $C \subseteq B_\omega \cap S_1' \cap S_2^\perp$, there exists a $^*$-homomorphism $\Psi:C \otimes \Z \rightarrow B_\omega \cap S_1' \cap S_2^\perp$ such that $\Psi(x \otimes 1_\Z) = x$ for all $x \in C$. 
		
		\item[(iii)]\label{lem:Zfacts4} If each $B_n$ is projectionless, then $B_\omega$ has stable rank one in $B_\omega^\sim$.
		\item[(iv)]\label{lem:Zfacts6} If $S \subseteq B_\omega$ is separable and self-adjoint, and $b \in (B_\omega \cap S')_+$, then for any $n\in \mathbb{N}$ there exists $c\in (B_\omega\cap S')_+$ with $c \leq b$ such that $n[c] \leq [b] \leq (n+1)[c]$ in $W(B_\omega\cap S')$.
		
	\end{itemize}
\end{lemma}
\begin{proof}
	Observe that (i) is the same as in \cite[Lemma 1.22.(i)]{BBSTWW} and follows as $\Z$ is strongly self-absorbing.
	
	For (ii), by Lemma \ref{lem:ActAsUnitNew}, there exists a positive contraction $h \in B_\omega \cap S_1' \cap S_2^\perp$ that acts as a unit on $T$. Set $S := S_1 \cup S_2 \cup T \cup \lbrace h \rbrace$. Let $\Phi:B_\omega \rightarrow\prod_{\omega} (B_n \otimes \Z)$ be the isomorphism from (i) with $\Phi(x) = x \otimes 1_\Z$ for all $x \in S$. Define a c.p.c.\ order zero map $\psi':\Z \rightarrow \prod_{\omega} (B_n \otimes \Z) $ by $\psi'(z) := h \otimes z$. By the choice of $h$, $\psi(1_\Z)$ acts as a unit on $T \otimes 1_\Z$ and the image of $\psi$ lies in $(B_\omega \otimes 1_\Z) \cap (S_1 \otimes 1_\Z)' \cap (S_2 \otimes 1_\Z)^\perp \cap (T \otimes 1_\Z)'$. Now set $\psi := \Phi^{-1} \circ \psi'$.
	
	For (ii'), by part (ii), there exists a c.p.c.\ order zero map $\psi:\Z \rightarrow  B_\omega \cap S_1' \cap S_2^\perp \cap T'$ such that $\psi(1_Z)$ acts as a unit on $C$. Since $\Z$ is nuclear, we can define a c.p.c.\ order zero map $\Psi:C \otimes \Z \rightarrow B_\omega \cap S_1' \cap S_2^\perp$ by $x \otimes z \mapsto x\psi(z)$. Let $x_1,x_2 \in C$ and $z_1,z_2 \in \Z$. Then
	\begin{align}
	\Psi(x_1 \otimes z_1)\Psi(x_1 \otimes z_1) &= x_1\psi(z_1)x_2\psi(z_2) \notag \\
	&= x_1x_2\psi(1_Z)\psi(z_1z_2) \notag \\
	&= x_1x_2\psi(z_1z_2) \notag \\
	&= \Psi(x_1x_2 \otimes z_1z_2),
	\end{align}
	where we have used the order zero identity in the second line. Hence, $\Psi$ is in fact a $^*$-homomorphism. Moreover, we have $\Psi(x \otimes 1_\Z) = x\psi(1_\Z) = x$ for all $x \in C$.
	
	Part (iii) follows by combining Theorem \ref{thm:ZstableProjectionless} with Proposition \ref{prop:SR1inUnitisationUltrapowers}.
	
	For (iv), let $C$ be the $\mathrm{C}^*$-algebra generated by $b$. By (ii'), there is a $^*$-homomorphism $\Psi:C \otimes \Z \rightarrow B_\omega \cap S'$. By \cite[Lemma 4.2]{Ro04}, there exists $e_n \in \Z_{+,1}$ with $n[e_n] \leq [1_\Z] \leq (n+1)[e_n]$. Set $c := \Psi(b \otimes e_n)$.  
\end{proof}

The following lemmas are crucial to the results of \cite[Section 5]{BBSTWW}. The proof of the first needs to be adapted slightly to the non-unital setting. 

\begin{lemma}[cf.\ {\cite[Lemma 2.1]{BBSTWW}}]
	\label{lem:Robertsr1New}
	Let $(B_n)_{n=1}^\infty$ be a sequence of $\mathcal Z$-stable $\mathrm C^*$-algebras and set $B_\omega :=\prod_\omega B_n$.
	Let $S \subseteq B_\omega$ be separable and self-adjoint, and let $d \in (B_\omega \cap S')_+$ be a contraction.
	Suppose that $x,f \in C:=B_\omega \cap S' \cap \{1_{B_\omega^\sim}-d\}^\perp$ are such that $xf=fx=0$, $f \geq 0$ and $f$ is full in $C$.
	Then $x$ is approximated by invertibles in $C^\sim$.
\end{lemma}
\begin{proof}
	By Lemma \ref{lem:ActAsUnitNew} (with $T:=\{x^*x,xx^*\},S_1:=S,S_2:=\{1_{B_\omega^\sim}-d,f\}$), we obtain a contraction $e \in C_+$ such that $xx^*,x^*x\vartriangleleft e$ and $ef=0$. Polar decomposition yields $ex=xe=x$. As in the proof of \cite[Lemma 2.1]{BBSTWW}, we may find a separable subalgebra $C_0$ of $C$ containing $x,e$, and $f$, such that $f$ is full in $C_0$. By Lemma \ref{lem:Zfacts}(ii'), there is a $^*$-homomorphism $\Psi: C_0 \otimes \Z \rightarrow C$ such that $\Psi(x \otimes 1_\Z) = x$ for all $x \in C_0$. By \cite[Lemma 2.1]{Rob16}, $x\otimes 1_{\mathcal Z}$ is a product of two nilpotent elements $n_1,n_2 \in C_0\otimes\mathcal Z$. It follows that $x = \Psi(x \otimes 1_\Z) = \Psi(n_1)\Psi(n_2)$ is the product of two nilpotent elements in $C$. If $y\in C$ is nilpotent and $\epsilon>0$, the operator $y+\epsilon 1_{C^\sim}$ is invertible in $C^\sim$ (with inverse $-\sum_{k=1}^N(-\epsilon)^{-k}y^{k-1}$, where $N\in\mathbb{N}$ satisfies $y^N=0$).  Therefore, $x$ can be approximated by invertible elements in $C^\sim$.
\end{proof}

\begin{lemma}[cf.\ {\cite[Lemma 2.2]{BBSTWW}}]\label{lem:GapInvertibles}
	Let $(B_n)_{n=1}^\infty$ be a sequence of $\mathcal Z$-stable $\mathrm C^*$-algebras and set $B_\omega :=\prod_\omega B_n$.
	Let $S \subseteq B_\omega$ be separable and self-adjoint, and let $d \in (B_\omega \cap S')_+$ be a contraction.
	Suppose that $x,s \in C:=B_\omega \cap S' \cap \{1_{B_\omega^\sim}-d\}^\perp$ are such that $xs=sx=0$ and $s$ is full in $C$.
	Then $x$ is approximated by invertibles in $C^\sim$.
\end{lemma}
\begin{proof}
	Let $C_0$ be the $\mathrm{C}^*$-subalgebra of $C$ generated by $x$ and $s$. By Lemma \ref{lem:Zfacts}(ii'), there exists a $^*$-homomorphism $\Psi:C_0 \otimes \Z \rightarrow C$ with $\Psi(y \otimes 1_\Z) = y$ for all $y \in C_0$.  Let $z_1,z_2 \in \mathcal Z_+$ be non-zero orthogonal elements. Set $s' := \Psi(s \otimes z_1)$ and $f:=\Psi(|s| \otimes z_2)$. As in the proof of \cite[Lemma 2.2]{BBSTWW}, it follows by Lemma \ref{lem:Robertsr1New} (with $s'$ in place of $x$) that $s'$ is approximated by invertibles in $C^\sim$. We finish the proof exactly as in the proof of \cite[Lemma 2.2]{BBSTWW} where we replace  \cite[Lemma 1.17]{BBSTWW} with Lemma \ref{NewEpsLemmaNew}, and \cite[Lemma 2.1]{BBSTWW} with Lemma \ref{lem:Robertsr1New}.
\end{proof}

\begin{lemma}[{cf.\ \cite[Lemma 5.4]{BBSTWW} and \cite[Lemma 2]{RS10}}]
	\label{lem:RSLem}
	Let $B$ be a separable, $\mathcal Z$-stable $\mathrm C^*$-algebra and let $A$ be a separable, unital $\mathrm C^*$-algebra. 
	Let $\pi:A\rightarrow B_\omega$ be a c.p.c.\ order zero map such that
	\begin{equation}
	C:=B_\omega \cap \pi(A)' \cap \{1_{B^\sim_\omega}-\pi(1_A)\}^\perp
	\end{equation}
	is full in $B_\omega$. Assume that every full hereditary subalgebra $D$ of $C$ satisfies the following: if $x \in D$ is such that there exist totally full elements $e_l,e_r \in D_+$ such that $e_l x = xe_r = 0$, then there exists a full element $s \in D$ such that $sx = xs = 0$.
	Let $e,f,f',\alpha, \beta \in C_+$ be such that
	\begin{equation}
	\label{eq:RSLem2Hyp1}
	\alpha \vartriangleleft e,\quad \alpha \sim \beta \vartriangleleft f, \quad\text{and} \quad f \sim f' \vartriangleleft e.
	\end{equation}
	Suppose also that there exist $d_e,d_f \in C_+$ that are totally full, such that
	\begin{align}
	\notag
	d_e &\vartriangleleft e, \quad d_e\alpha = 0, \quad \text{and} \\
	d_f &\vartriangleleft f, \quad d_f\beta = 0.\label{eq:RSLem2Hyp2}
	\end{align}
	Then there exists $e' \in C_+$ such that
	\begin{align}
	\alpha \vartriangleleft e' \vartriangleleft e, \quad \text{and} \quad \alpha+e' \sim \beta+f.
	\end{align}
\end{lemma}

\begin{proof}[Remarks]
	The proof of \cite[Lemma 2]{RS10} does not assume unitality.
	The proof from \cite{BBSTWW} is still valid after replacing \cite[Lemma 1.17]{BBSTWW} with Lemma \ref{NewEpsLemmaNew}, \cite[Lemma 2.2]{BBSTWW} with Lemma \ref{lem:GapInvertibles}, and using $1_{B^\sim_\omega}$ in place of $1_{B_\omega}$. 
\end{proof}

The following lemma concerns the interplay between strict comparison and ultraproducts in the non-unital setting.  
\begin{lemma}[cf.\ {\cite[Lemma 1.23]{BBSTWW}}] \label{lem:StrictCompLimTraces} 
	Let $(B_n)_{n=1}^\infty$ be a sequence of $\mathrm C^*$-algebras with $T(B_n)$ non-empty and set $B_\omega:=\prod_\omega B_n$. Suppose each $B_n$ has strict comparison of positive elements with respect to bounded traces.
	Then $B_\omega$ has strict comparison of positive elements with respect to limit traces, in the following sense:
	If $a,b \in M_k(B_\omega)_+$ for some $k\in\mathbb N$ satisfy $d_\tau(a) < d_\tau(b)$ for all $\tau$ in the weak$^{*}$-closure of $T_\omega(B_\omega)$, then $a \preceq b$.
\end{lemma} 

\begin{proof}[Remarks]
	The proof is identical to that of \cite[Lemma 1.23]{BBSTWW}. However, it is important to note that, in the non-unital case, we do not necessarily have that $\overline{T_\omega(B_\omega)}^{w*} \subseteq T(B_\omega)$, as the later need not be closed. Indeed, we may have $0 \in \overline{T_\omega(B_\omega)}^{w*}$ in which case $d_\tau(a) < d_\tau(b)$ cannot hold for all $\tau \in \overline{T_\omega(B_\omega)}^{w*}$.     
\end{proof}

Finally, we record a technical lemma needed for the proof of the main theorem of the property (SI) section. 

\begin{lemma}[cf.\ {\cite[Lemma 4.9]{BBSTWW}}] \label{lem:relSItrick}
	Let $B$ be a simple, separable,  $\mathrm{C}^*$-algebra with $Q\widetilde T(B) = \widetilde{T}_b(B)\neq 0$. Suppose $B$ has strict comparison of positive elements by bounded traces. Let $A$ be a separable, unital $\mathrm{C}^*$-algebra and let $\pi:A\rightarrow B_\omega$ be a c.p.c.\ order zero map. Let $a \in A_+$ be a positive contraction of norm $1$. Then there exists a countable set $S \subseteq A_+\setminus \{0\}$ such that the following holds:
	If $e,t,h \in (B_\omega \cap \pi(A)' \cap \lbrace 1_{B_\omega^\sim}-\pi(1_A)\rbrace^\perp)_+$ are contractions such that
	\begin{equation} e \in J_{B_\omega}\quad \text{and} \quad h \vartriangleleft t,
	\end{equation}
	and if for all $b \in S$, there exists $\gamma_b > 0$ such that
	\begin{equation}\label{eq:relSItrickTr}
	\tau(\pi(b)h) > \gamma_b,\quad \tau\in T_\omega(B_\omega), \end{equation}
	then there exists a contraction $r \in B_\omega$ such that
	\begin{equation} \label{eq:relSItrickrDef}
	\pi(a)r=tr=r \quad \text{and} \quad r^*r=e. 
	\end{equation}
\end{lemma}
\begin{proof}[Remarks]
	The proof of \cite[Lemma 4.9]{BBSTWW} works in our situation using Lemma \ref{lem:StrictCompLimTraces} in place of \cite[Lemma 1.23]{BBSTWW}.
\end{proof}

\end{document}